\documentclass{article}  
\usepackage{amsmath,amsthm,amssymb,amscd}

\usepackage{graphicx}
\usepackage{epstopdf}
\usepackage{color}
\usepackage{url}

\numberwithin{equation}{section}

\theoremstyle{definition}\newtheorem{definition}{Definition}
\theoremstyle{plain}\newtheorem{theorem}[definition]{Theorem}
\theoremstyle{definition}
\theoremstyle{definition}\newtheorem{proposition}[definition]{Proposition}
\theoremstyle{definition}
\theoremstyle{definition}

\newcommand{\R}{\mathbb{R}}
\newcommand{\N}{\mathbb{N}}

\title{Using Landweber method to quantify source conditions - a numerical study.}

\author{Daniel Gerth 
\footnote{Technische Universit\"at Chemnitz,
Faculty of Mathematics, D-09107 Chemnitz, Germany, daniel.gerth@mathematik.tu-chemnitz.de}}

\begin{document}
\maketitle
\begin{abstract}Source conditions of the type $x^\dag \in\mathcal{R}((A^\ast A)^\mu)$ are an important tool in the theory of inverse problems to show convergence rates of regularized solutions as the noise in the data goes to zero. Unfortunately, it is rarely possible to verify these conditions in practice, rendering data-independent parameter choice rules unfeasible. In this paper we show that such a source condition implies a Kurdyka-\L{}ojasiewicz inequality with certain parameters depending on $\mu$. While the converse implication is unclear from a theoretical point of view, we demonstrate how the Landweber method in combination with the Kurdyka-\L{}ojasiewicz inequality can be used to approximate $\mu$ and conduct several numerical experiments. We also show that the source condition implies a lower bound on the convergence rate which is of optimal order and observable without the knowledge of $\mu$.\end{abstract}
\section{Introduction}\label{sec:intro}
Let $A:X\rightarrow Y$ be a bounded linear operator between Hilbert spaces $X$ and $Y$. A vast class of mathematical problems boils down to the solution of an equation
\begin{equation}\label{eq:problem}
Ax=y.
\end{equation}
In inverse problems, where $A$ is assumed to be ill-posed, i.e., $\mathcal{R}(A)\neq\overline{\mathcal{R}(A)}$, typically only noisy data $y^\delta$ with $\|y-y^\delta\|\leq\delta$, $\delta>0$ are available. In this case, one will not obtain the exact solution $x^\dag$ to \eqref{eq:problem}, but may only hope to find a reasonable approximation $x^\delta$ to $x^\dag$. One of the main questions in the theory of inverse problems is the one of rates of convergence, i.e., finding a function $\varphi:[0,\infty)\rightarrow[0,\infty)$, $\varphi(0)=0$, such that
\begin{equation}\label{eq:crates}
\|x^\dag-x^\delta\|\leq \varphi(\delta)\quad \forall \,0<\delta\leq \delta_0.
\end{equation}
It is known that in general such a function $\varphi$ cannot be found as the convergence may be arbitrarily slow; see, e.g., \cite{EHN96}. Instead, the class of solutions has to be restricted in order to obtain a convergence rate. A classical but still widespread condition for this is to assume that
\begin{equation}\label{eq:sc}
x^\dag\in\mathcal{R}((A^\ast A)^\mu)
\end{equation}
for some $\mu>0$. While this condition is clear and applicable from the theoretical point of view, its practical usefulness is limited. For given $A$ and noisy data $y^\delta$ it is seldom possible to find the correct smoothness parameter $\mu$. Even when the exact solution $x^\dag$ is known it is often not an easy task to determine $\mu$.

We demonstrate the importance of the knowledge of $\mu$ exemplarily by recalling the method of classical Tikhonov regularization. Define
\begin{equation}\label{eq:tikh}
T_\alpha^\delta(x):=\|Ax-y^\delta\|^2+\alpha\|x\|^2.
\end{equation}
The minimizer $x_\alpha^\delta:=\mathrm{argmin}_{x\in X} T_\alpha^\delta(x)$ is called Tikhonov-regularized solution to \eqref{eq:problem} under the noisy data $y^\delta$. The convergence of Tikhonov regularization has been extensively studied, we mention for example the monographs \cite{EHN96,Louis,TikhonovArsenin1977}. All of them contain a version of the following proposition.
\begin{proposition}\label{thm:tikh}
Let $x^\dag\in\mathcal{R}((A^\ast A)^\mu)$ for some $0<\mu\leq1$. Then the minimizers $x_\alpha^\delta$ of \eqref{eq:tikh} satisfy
\begin{equation}\label{eq:rate_tikh}
\|x_\alpha^\delta-x^\dag\|\leq C\delta^{\frac{2\mu}{2\mu+1}}
\end{equation}
with $0<C<\infty$ provided the regularization parameter $\alpha$ is chosen via
\begin{equation}\label{eq:apriori}
c_1 \delta^{\frac{2}{2\mu+1}}\leq \alpha\leq c_2\delta^{\frac{2}{2\mu+1}}
\end{equation}
where $c_1\leq c_2$ are positive constants.
\end{proposition}
Thus, $\mu$ not only essentially determines the convergence rate but is also crucial to obtain it in practice. We mention that the rate \eqref{eq:rate_tikh} can also be obtained via the discrepancy principle for the choice of $\alpha$; see, e.g., \cite{EHN96,Morozov}. The idea behind this method is to choose $\alpha$ such that $\delta \leq \|Ax_\alpha^\delta-y^\delta\|\leq \tau \delta$ for some $\tau>1$. While the discrepancy principle is applicable independent of the knowledge of $\mu$, it only yields the rate \eqref{eq:rate_tikh} for $0<\mu\leq\frac{1}{2}$. Therefore, if $\mu$ is unknown, one might apply the discrepancy principle in instances where this method is not justified anymore.

\section{The Kurdyka-\L{}ojasiewicz inequality}
A \L{}ojasiewicz-type inequality will be the main tool used for our method. It can be formulated in complete metric spaces and we will do so for the moment. We temporarily consider the abstract problem
\[
f(x)\rightarrow \min_{x\in X}
\] 
where $X$ is a complete metric space with metric $d_X(x,y)$ and $f:X\rightarrow \R\cup \{\infty\}$ is lower semicontinuous. We need some definitions.

We denote with
\begin{equation}
[t_1\leq f\leq t_2]:=\{x\in X: t_1\leq f(x)\leq t_2\}
\end{equation}
the level-set of $f$ to the levels $t_1\leq t_2$. For $t_1=t_2=:t$ we write $[f=t]$ and analogously for the relations $\leq$ and $\geq$. 
For any $x\in X$ the distance of $x$ to a set $S\subset X$ is denoted by
\begin{equation}
\mathrm{dist}(x,S):=\inf_{y\in S} d_X(x,y).
\end{equation}
With this we recall the Hausdorff distance between sets,
\begin{equation}
D(S_1,S_2):=\max\left\lbrace\sup_{x\in S_1}\mathrm{dist}(x,S_2),\sup_{x\in S_2} \mathrm{dist}(x,S_1) \right\rbrace.
\end{equation}
For $x\in X$ we define the strong slope as
\begin{equation}
|\nabla f|(x):=\limsup_{y\rightarrow x}\frac{\max\{f(x)-f(y),0\}}{d(x,y)}.
\end{equation}

The multi-valued mapping $F:X \rightrightarrows Y$ is called k-metrically regular if for all $(\bar x,\bar y)\in \mathrm{Graph}(F)$ there exist $\epsilon_1,\epsilon_2>0$ such that for all $(x,y)\in B(\bar x,\epsilon_1)\times B(\bar y,\epsilon_2)$ it is
\[
\mathrm{dist}(x,F^{-1}(y))\leq k \,\mathrm{dist}(y,F(x)).
\] 

\begin{definition}
We call the function $\varphi:[0,\bar r)\rightarrow \R$ a smooth index function if $\varphi\in C[0,\bar r)\cap C^1(0,\bar r)$, $\varphi(0)=0$ and $\varphi^\prime(x)>0$ for all $x\in(0,\bar r)$. We denote the set of all such $\varphi$ with $\mathcal{K}(0,\bar r)$.
\end{definition}
In the optimization literature such functions are called \textit{desingularizing functions}. We now present the main theorem our work is based on, taken from \cite{BDLM08}.

\begin{proposition}{\cite[Corollary 7]{BDLM08}}\label{thm:cor7}
Let $f:X\rightarrow \R$ be continuous, strongly slope-regular, i.e., $|\nabla f|(x)=|\nabla (-f)|(x)$, $(0,r_0)\subset f(X)$ and $\varphi\in \mathcal{K}(0, r_0)$. Then the following assumptions are equivalent and imply the non-emptiness of the level set $[f=0]$:
\begin{itemize}
\item $\varphi\circ f$ is $k$-metrically regular on $[0<f<r_0]\times (0,\varphi(r_0))$,
\item for all $r_1,r_2\in (0,r_0)$ it is
\begin{equation}\label{eq:rate_equation_c7}
D([f=r_1],[f=r_2])\leq k |\varphi(r_1)-\varphi(r_2)|,
\end{equation}
\item for all $x\in[0<f<r_0]$ it holds that
\begin{equation}\label{eq:kl_c7}
|\nabla(\varphi\circ f)|(x)\geq \frac{1}{k}.
\end{equation}
\end{itemize}
\end{proposition}
While we added the first item for the sake of completeness, the second and third properties are what is interesting to us. Eq. \eqref{eq:rate_equation_c7} states that the (Hausdorff)-distance between the level-sets is governed by the value of the index function $\varphi$ at the value of the level. Assume we have two values $x_1,x_2\in X$ at hand such that $f(x_1)=c_1\delta$, $f(x_2)=c_2\delta$ and that $\varphi$ is concave. Then, by definition,
\[
d_X(x_1,x_2)\leq D([f=c_1\delta],[f=c_2\delta])\leq k|\varphi(c_1\delta)-\varphi(c_2\delta)|
\] 
and without loss of generality $c_1>c_2$ yields 
\[
d_X(x_1,x_2)\leq kc_1\varphi(\delta).
\]
This is a generalization of the convergence rates \eqref{eq:crates} and as such of high interest in the context of inverse problems.

The last item \eqref{eq:kl_c7} is called Kurdyka-\L{}ojasiewicz inequality. The result by \L{}ojasiewicz \cite{loja,loja2}, who originally considered real analytic functions, states in principle that for a $C^1$-function $f$ the quantity
\[
|f(x)-f(\bar x)|^\theta\|\nabla f\|^{-1}
\]
remains bounded around any point $\bar x$. Let for example $f(x)=x^2$ and $\bar x=0$. Then one easily sees that 
\[
|f(x)-f(\bar x)|^\theta\|\nabla f\|^{-1}=\frac{x^{2\theta}}{2x}=\frac{1}{2}
\]
for $\theta=\frac{1}{2}$.

Kurdyka generalized the inequality \eqref{eq:kl_c7} in \cite{Kurdyka} and later other versions showed up, see for example \cite{Bolte,BolteEtAl07,BolteEtAl072}. Useful in our context appears to be the formulation from, e.g., \cite{Bolte},
\begin{equation*}
\varphi^\prime(f(x)-f(\bar x))\mathrm{dist}(0,\partial f(x))\geq 1,
\end{equation*}
where $\varphi$ is as in Proposition \ref{thm:cor7}, $f$ is convex, $\bar x$ is a minimizer of $f$, and $x$ is sufficiently close to $\bar x$. The convexity of $f$ is not necessary as one can move to further generalized subdifferentials e.g., Frechet subdifferentials or the limiting subdifferential. 
If $f$ is differentiable, then we can write
\begin{equation}\label{eq:kl_grad}
\varphi^\prime(f(x)-f(\bar x))\|\nabla f(x)\|\geq 1.
\end{equation}
As we will see in the following one dimensional example, the function $\varphi$ is, roughly speaking, such that $\varphi(f(x)-f(\bar x))$ is essentially a linear function. 

Let $\varphi$ be a smooth index function and assume $f:\R\rightarrow\R$ to be positive, convex, and $f(0)=0$. If $\frac{d}{dx}\varphi(f(x))=\varphi^\prime(f(x))\frac{d}{dx}f(x)=\mathrm{const}$, then $\varphi(f(x))$ must be an (affine) linear function, see the example of a quadratic function above. Please note here the similarity between the chain rule and the Kurdyka-\L{}ojasiewicz inequality \eqref{eq:kl_grad}.
The \L{}ojasiewicz property generalizes this and basically describes the flatness of the functional that is to be minimized around the solution. The flatter this is, the slower the speed of convergence.

\section{Lojasiewicz inequality and source condition in Hilbert spaces}
In Hilbert spaces we may rewrite Proposition \ref{thm:cor7} as follows.
\begin{proposition}\label{thm:equivalentstuff}
Let $f:X\rightarrow \R$ be continuous and $(0,r_0)\subset f(X)$. Let $\varphi$ be a  smooth index function. Then the following assumptions are equivalent.
\begin{itemize}
\item for all $x_1,x_2\in [0<f<r_0]$
\begin{equation}\label{eq:rate_equation}
\|x_1-x_2\|\leq k |\varphi(f(x_1))-\varphi(f(x_2))|
\end{equation}
\item for all $x\in[0<f<r_0]$
\begin{equation}\label{eq:kl}
\varphi^\prime(f(x)-f(\bar x))\|\nabla f(x)\|\geq \frac{1}{k}
\end{equation}
where $\bar x$ minimizes $f$.
\end{itemize}
\end{proposition}

As mentioned before, setting $x_1:=x^\delta$ and $x_2:=x^\dag$ with $x^\delta$ and $x^\dag$ from Section \ref{sec:intro}, immediately yields convergence rates of the form \eqref{eq:crates}. It is therefore interesting how \eqref{eq:kl} relates to the source condition \eqref{eq:sc}. We will now show that the classical source condition $x^\dag\in\mathcal{R}((A^\ast A)^\mu)$ implies $\eqref{eq:kl}$ for the squared residual $f(x)=\|Ax-y\|^2$.

Consider the noise-free least-squares minimization
\begin{equation}\label{eq:leastsquares}
f(x):=\|Ax-y\|^2\rightarrow \min_x
\end{equation}
where $y=Ax^\dag$ and as before $x^\dag\in X$ is the exact solution to \eqref{eq:problem} in the Hilbert space $X$. Since \eqref{eq:leastsquares} is differentiable with $\nabla f=2A^\ast(Ax-y)$ we may verify the smooth KL-inequality \eqref{eq:kl}. In order to do so we assume, analogously to \eqref{eq:sc}, 
\begin{equation}\label{eq:sc2}
x-x^\dag =(A^\ast A)^\mu w
\end{equation}
for some $\mu>0$ and $w$ in $X$. The only other ingredient is the well-known interpolation inequality
\begin{equation}\label{eq:interpol}
\|(A^\ast A)^r x\|\leq \|(A^\ast A )^qx\|^{\frac{r}{q}}\|x\|^{1-\frac{r}{q}},
\end{equation}
for all $q>r\geq 0$; see \cite{EHN96}. We are looking for a function $\varphi(t)=ct^{\kappa+1}$ with unknown exponent $-1<\kappa<0$. Then $\varphi^\prime(t)=c(1+\kappa)t^\kappa$. Thus \eqref{eq:kl} reads, with $\bar x=x^\dag$,
\[
f(x)^\kappa\|\nabla f(x)\|\geq c
\]
or equivalently
\[
f(x)^{-\kappa}\leq c\|\nabla f(x)\|.
\]
With $f(x)$ as above and $[\nabla f](x)=2A^\ast(Ax-y)$ we have

\begin{align*}
\left(\|Ax-y\|^2\right)^{-\kappa}&=\left(\|A(x-x^\dag)\|\right)^{-2\kappa}\\
&=\left(\|A(A^\ast A)^\mu w\|\right)^{-2\kappa}\\
&=\left(\|(A^\ast A)^{\mu+\frac{1}{2}} w\|\right)^{-2\kappa}\\
&\leq\left( \|(A^\ast A)^{\mu+1} w\|^{\frac{\mu+\frac{1}{2}}{\mu+1}}\|w\|^{1-\frac{\mu+\frac{1}{2}}{\mu+1}}\right)^{-2\kappa}\\
&=\|(A^\ast A)^{\mu+1}w\|^{-2\kappa\frac{\mu+\frac{1}{2}}{\mu+1} }\|w\|^{-\frac{\kappa}{\mu+1}}\\
&=\|A^\ast(Ax-y) \|^{-2\kappa\frac{\mu+\frac{1}{2}}{\mu+1} }\|w\|^{-\frac{\kappa}{\mu+1}}.
\end{align*}
The correct $\kappa$ is the one satisfying 
\[
-2\kappa\frac{\mu+\frac{1}{2}}{\mu+1}=1,
\]
i.e.,
\[
\kappa=-\frac{\mu+1}{2\mu+1}.
\]
Thus our original function $\varphi(t)=c t^{\kappa+1}$ is given by
\begin{equation}\label{eq:vsrphi_sc_mu}
\varphi(t)=c t^{\frac{\mu}{2\mu+1}}.
\end{equation}
In terms of \eqref{eq:rate_equation} this yields for $f(x)=\delta^2$ from \eqref{eq:leastsquares} the rate $\varphi(\delta^2)=c \delta^{\frac{2\mu}{2\mu+1}}$ which is well-known from the theory of optimal convergence rates. Note that we also obtain $\|w\|^{-\frac{\kappa}{\mu+1}}=\|w\|^{\frac{1}{2\mu+1}}$ which is the standard estimate for this term, see, e.g., \cite{EHN96,Louis}.

Unfortunately, at the moment it is not clear whether a \L{}ojasiewicz-inequality \eqref{eq:kl} implies that a source condition \eqref{eq:sc} holds. Nevertheless, we will see that often it is possible to find an approximation to $\mu$ in \eqref{eq:sc} by looking for a function fulfilling \eqref{eq:kl}.

\section{Observable lower bounds for the reconstruction error}

We have seen that the source condition \eqref{eq:sc} implies a \L{}ojasiewicz inequality with $\varphi$ from \eqref{eq:vsrphi_sc_mu}, which implies a convergence rate \eqref{eq:rate_equation} i.e., an upper bound on the distance between to elements in $X$ based on the distance of their images under $A$ in $Y$. In the following we show that the source condition also implies a lower rate, and that upper and lower boundary on the convergence rate are of the same order, disagreeing only in a constant factor. Even more, this lower bound is easily computable in practice, and no information on $\mu$ is required to do so.

\begin{theorem}
Let $A:X\rightarrow Y$ be a compact linear operator between Hilbert spaces $X$ and $Y$. Let $x^\dag\in X$ fulfill a source condition \eqref{eq:sc} for some $\mu>0$ and let $\| Ax-Ax^\dag||$ sufficiently small. Then, whenever $\nabla (\|Ax-Ax^\dag\|^2)\neq 0$, it is
\begin{equation}\label{eq:doublebound}
c_1  \varphi(\|Ax-Ax^\dag\|^2) \leq \|x-x^\dag\| \leq c_2 \varphi(\|Ax-Ax^\dag\|^2)
\end{equation}
with constants $0<c_1<c_2<\infty$ and $\varphi(t)=t^{\frac{\mu}{2\mu+1}}$ from \eqref{eq:vsrphi_sc_mu}.
\end{theorem} 
\begin{proof}
The upper bound follows from Proposition \ref{thm:equivalentstuff}. To obtain the lower bound, we note that
\begin{equation}\label{eq:rate_uplow}
\|Ax-Ax^\dag\|^2=\langle x-x^\dag,A^\ast (Ax-Ax^\dag)\rangle\leq \|x-x^\dag\|\,\|A^\ast(Ax-Ax^\dag)\|
\end{equation}
and hence
\begin{equation}\label{eq:lb}
\frac{\|Ax-Ax^\dag\|^2}{\|A^\ast(Ax-Ax^\dag)\|}\leq \|x-x^\dag\|.
\end{equation}
We now use the interpolation equality \eqref{eq:interpol} in the denominator. Together with \eqref{eq:sc2} it follows that
\begin{align*}
\|A^\ast(Ax-Ax^\dag)\|&=\|(A^\ast A)^{\mu+1} w\|\\
&\leq \|(A^\ast A)^{\frac{1}{2}+\mu}w\|^{\frac{\mu+1}{\mu+\frac{1}{2}}}\|w\|^{1-\frac{\mu+1}{\mu+\frac{1}{2}}}\\
&= \|Ax-Ax^\dag\|^{\frac{2\mu+2}{2\mu+1}}\|w\|^{-\frac{1}{2\mu+1}}
\end{align*}
Inserting this into \eqref{eq:lb} yields
\begin{align*}
\|x-x^\dag\|\geq \|Ax-Ax^\dag\|^{\frac{2\mu}{2\mu+1}}\|w\|^{\frac{1}{2\mu+1}}
\end{align*}
\end{proof}
Equation \eqref{eq:lb} is not limited to Hilbert spaces. In particular when only $X$ is a Banach spaces with dual space $X^\ast$, then one simply needs to use the dual product $\langle\cdot,\cdot\rangle_{X^\ast\times X}$ instead of the scalar product. Another advantage of \eqref{eq:lb} is that the lower bound is often easily observable since it only contains the (squared) residual and the gradient. Since under the source condition \eqref{eq:sc2} also
\[
c_1 \|Ax-Ax^\dag\|^{\frac{2\mu}{2\mu+1}}\leq \frac{\|Ax-Ax^\dag\|^2}{\|A^\ast(Ax-Ax^\dag)\|}\leq c_2\|Ax-Ax^\dag\|^{\frac{2\mu}{2\mu+1}},
\]
it appears tempting to try to estimate $\mu$ also from the lower bound, but it can easily be seen that this leads to the same regression problem than the upper bound which we will discuss in the next section. However, we found it useful to check the credibility of the estimated $\mu$. In particular, as we will show in the numerical examples that the lower bound it allows to detect when the noise takes over or when the discretization is insufficient.

\section{From inequality to algorithm}\label{sec:alg}
Many iterative methods for the minimization of a functional $f(x)$ use the gradient of $f$ at the current iterate to update the unknown. For $f(x)=\|Ax-y\|^2$, the gradient $A^\ast (Ax-y)$ includes the calculation of the residual $Ax-y$. Therefore, any such iterative method allows to store and compute the values $\|Ax_k-y\|^2$ and $\|A^\ast(Ax_k-y)\|$ as the iteration goes on. Since the \L{}ojasiewicz inequality relates both values through an unknown one dimensional function, we may use the acquired values of gradients and residuals to estimate that function via a regression.

Our particular choice of the iterative algorithm is the Landweber method, but the algorithm is not restricted to this choice. This is a well-known method to  minimize \eqref{eq:leastsquares}; see, e.g., \cite{EHN96,Louis}. Starting from an initial guess $x_0$, it consists in iterating
\begin{equation}\label{eq:lw_free}
x_{k+1}=x_k-\beta A^\ast (Ax_k-y)
\end{equation}
for $k=0,1,\dots,K$ where $0<\beta< \frac{2}{\|A\|^2}$ and $K$ is the stopping index.

Let $\{x_k\}_{k=1,\dots,K}$ be the sequence of iterates obtained from the Landweber method. During the iterations we compute and store the norms of both residual and gradient, i.e., we acquire two vectors
\[
R:=(\|Ax_1-y\|,\|Ax_2-y\|,\|Ax_3-y\|,\dots,\|Ax_K-y\|)^T
\]
and
\[
G:=(\|A^\ast (Ax_1-y)\|,\|A^\ast (Ax_2-y)\|,\|A^\ast (Ax_3-y)\|,\dots,\|A^\ast(Ax_K-y)\|)^T.
\]
If the \L{}ojasiewicz property \eqref{eq:kl_grad} holds with $\bar x=x^\dag$, then there is a smooth index function $\varphi$ such that 
\begin{equation}\label{eq:kl_ls}
\varphi^\prime(R_i^2)\cdot G_i\geq 1 \quad \forall i=1,\dots,K,
\end{equation}
where the constant $k$ appearing in \eqref{eq:kl_grad} has been moved into the function $\varphi^\prime$.

As we have seen previously, the source condition \eqref{eq:sc} implies $\varphi(t)=ct^{\frac{\mu}{2\mu+1}}$ and hence $\varphi^\prime(t)=ct^{-\frac{\mu+1}{2\mu+1}}$. Setting
\[
\gamma:=\frac{2\mu+2}{2\mu+1},
\]
we have
\[
\varphi^\prime(R_i^2)=c(R_i^2)^{-\frac{\mu+1}{2\mu+1}}=cR_i^{-\gamma}\quad \forall i=1,\dots,K,
\]
and therefore we obtain
\begin{equation}\label{eq:regression_pre}
\frac{R_i^\gamma}{c}\leq G_i \quad \forall i=1,\dots,K,
\end{equation}
where we have combined the upper and lower constant in one single constant. In practice of course we do not know $\gamma$ or the constant $c$, but we have many data pairs $(R_i,G_i)$, $i=1,\dots,K$. Therefore, we may use the data to estimate $\gamma$ and $c$ via a regression approach. To this end, we take the $\log$ of \eqref{eq:regression_pre} which yields, replacing the inequality with an equality,
\begin{equation}\label{eq:regression}
\gamma\log(R_i)-\log(c)= \log(G_i) \quad \forall i=1,\dots,K.
\end{equation}
This is a linear regression problem for the variables $c_{l}:=\log(c)$ and $\gamma$. We write it in matrix form
\begin{equation}\label{eq:reg_matrix}
A_K\begin{pmatrix} \gamma \\ c_l\end{pmatrix}= b_K
\end{equation}
where 
\[
A_K:=\begin{pmatrix}
\log(R_1) &-1\\
 \log(R_2) &-1\\
\vdots & \vdots \\
 \log(R_K) &-1
\end{pmatrix}, \quad b_K=\begin{pmatrix}
\log(G_1)\\\log(G_2)\\ \vdots \\\log(G_K)
\end{pmatrix}.
\]

It is well known (see, e.g., \cite{shao}) that the best estimator, i.e., the estimator yielding the least error variance, for $[\gamma,c_l]^T$ in \eqref{eq:reg_matrix} is
\[
\begin{pmatrix}
  \hat\gamma\\\hat c_l
\end{pmatrix}= (A_K^TA_K)^{-1}A_K^Tb_K.
\]
This immediately yields $\mu_K=\frac{2-\hat\gamma}{2\hat\gamma-2}$ and $c_K=\exp(\hat c_{l})$. 

In our numerical experiments we found that it is more appropriate to do a regression after each iterate instead of just one at the end. We therefore define $\mu_k$ as estimates for $\mu$ when only the first $k$ rows of $A_K$ and $b_K$ are used. Analogously we have $c_k$ as estimated constant after $k$ iterations. This allows to track the smoothness parameter throughout the iteration. As shown in the next section, $\mu_k$ and $c_k$ are not constant but often vary as $k$ increases. However, we often find intervals $k_1,\dots,k_2$, $1\leq k_1<k_2\leq K$ where both $\mu_k$ and $c_k$ are approximately stable. In this region we observed that the parameters were estimated reasonably well. Since we do not have a fully automated system yet we pick the final estimate for $\mu$ by hand.

\section{Numerical experiments}
Using the method suggested in Section \ref{sec:alg} we now try to estimate the smoothness parameter $\mu$ for various linear ill-posed problems \eqref{eq:problem}. We give two types of plots. In the first type we plot the estimated $\mu$ against the Landweber iteration number. This should yield a constant value for $\mu$. In the second type of plots, we give the lower bound from \eqref{eq:lb}. If available, we also plot the measured reconstruction error $||x^\dag-x^\delta||$ and for the diagonal examples we show the upper bound from \eqref{eq:rate_uplow}. All of these functions are plotted against the residual $||Ax-y||$ in a loglog-plot. This should yield a linear function whose constant slope is again a measure for $\mu$.

\subsection{Benchmark case: diagonal operators}\label{sec:diag}
We start with the (mostly) academic example of a diagonal operator. While this is certainly not a particularly practical situation, it allows to see the chances and limitations of our approach due to the direct control of the parameters and low computational cost.

Let $A:\ell^2\rightarrow \ell^2$, $A: (x_1,x_2,x_3,\cdots)\mapsto (\sigma_1x_1,\sigma_2x_2,\sigma_3 x_3,\dots)$ for $\sigma_i\in \R$, $i\in\N$. To make the example even more academic, let $\sigma_i=i^{-\beta}$ for some $\beta>0$ and assume additionally that $x^\dag$ is given as $x_i^\dag=i^{-\eta}$, $i\in\N$, with $\eta>0$. Following \cite[Proposition 3.13]{EHN96} we have for a compact linear operator $A$ between Hilbert spaces $X$ and $Y$ with singular system $\{\sigma_i,u_i,v_i\}_{i=1}^\infty$ that
\begin{equation}\label{eq:verif_smoothness}
x\in \mathcal{R}((A^\ast A)^\mu)\quad \Leftrightarrow \quad \sum_{i=1}^\infty \frac{|\langle Ax,u_i\rangle|^2}{\sigma_i^{2+4\mu}}<\infty.
\end{equation}
 One quickly verifies that therefore $x\in\mathcal{R}((A^\ast A)^\mu)$ for $\mu\leq\frac{2\eta-1}{4\beta}-\epsilon$ and any small $\epsilon>0$. Numerically it appears unfeasible to determine $\mu$ to such a high precision. Therefore we give $\mu_{exact}=\frac{2\eta-1}{4\beta}$ although the source condition is just barely not satisfied in this case. Nevertheless, it gives an excellent fit for the diagonal operators where we have full control over the smoothness properties.

We present figures for different combinations of $\eta$ and $\beta$. More precisely, we have $\eta=1$ and $\beta=2.5$ in Figure \ref{fig:diag1}, $\eta=2$ and $\beta=2$ in Figure \ref{fig:diag2}, $\eta=2$ and $\beta=1$ in Figure \ref{fig:diag3}, and $\beta=1.5$ and $\eta=3$ in Figure \ref{fig:diag4}. After a burn-in time between 50 and 100 iterations a fairly accurate estimate of $\mu$ is achieved in all the examples. Except for the case $\eta=2$ and $\beta=1$, $\mu$ even seems to converge to the exact value as the iterations increase. The ``exception'' is easily explained when looking the reconstruction errors and the observed lower bound. We show the results for $\eta=2$ and $\beta=2$ in Figure \ref{fig:diag2LB}, where we have an almost perfect match between the reconstruction error, its lower and its upper bound. For the seemingly critical case $\eta=2$ and $\beta=1$ the same plot, shown in Figure \ref{fig:diag3LB}, reveals the the reason for the diverging estimate of $\mu$. In Figure \ref{fig:diag3LB}, the lower bound and the measured reconstruction error have two distinct phases with different slopes and a small transitioning phase. In the first phase, where the residuals are large, we have the slope corresponding to the good estimate of $\mu$ in Figure \ref{fig:diag3}. In the second phase, corresponding to the small residuals, the slope is approximately 1. What happens is that the discretization is too small, such that at some point the algorithm sees the full matrix $A$ as a well-posed operator from $\R^n$ to $\R^n$ instead of the discretization of the ill-posed problem $A$. Indeed, if we increase the discretization level, we obtain a plot with the correct asymptotics similar to Figure \ref{fig:diag1} and \ref{fig:diag3}.

Let us also check the algorithm when the source condition \eqref{eq:sc} is not strictly fulfilled or even violated. To this end we first let $\sigma_i=i^{-3/2}$ and $x_i=e^{-i}$, i.e., $x^\dag$ fulfills the source condition for all $\mu>0$. Over the iterations, $\mu$ starts large, approaches a value of about $\mu_{min}=0.5$ and then slowly starts to grow again; see Figure \ref{fig:expon}. The constant behaves similarly. The \L{}ojasiewicz inequality and therefore our algorithm seem to capture some kind of ``maximal'' smoothness which in this example is not adequately described with the source condition \eqref{eq:sc} for any $\mu>0$.
 We now move the exponential to the operator, i.e., consider the case $x_i=i^{-2}$ and $\sigma_i=e^{-i}$ such that $x^\dag$ fails to fulfill the source condition \ref{eq:sc} for any $\mu>0$. The result, shown in Figure \ref{fig:expon_op}, reveals a new pattern. Both $\mu$ and $c$ describe a sinusoidal graph. In neither of the two examples we found a section where $c$ and $\mu$ remain approximately stable. We conclude that the lack of such stability indicates the violation of the source condition \eqref{eq:sc}.

\begin{figure}
\includegraphics[width=\textwidth]{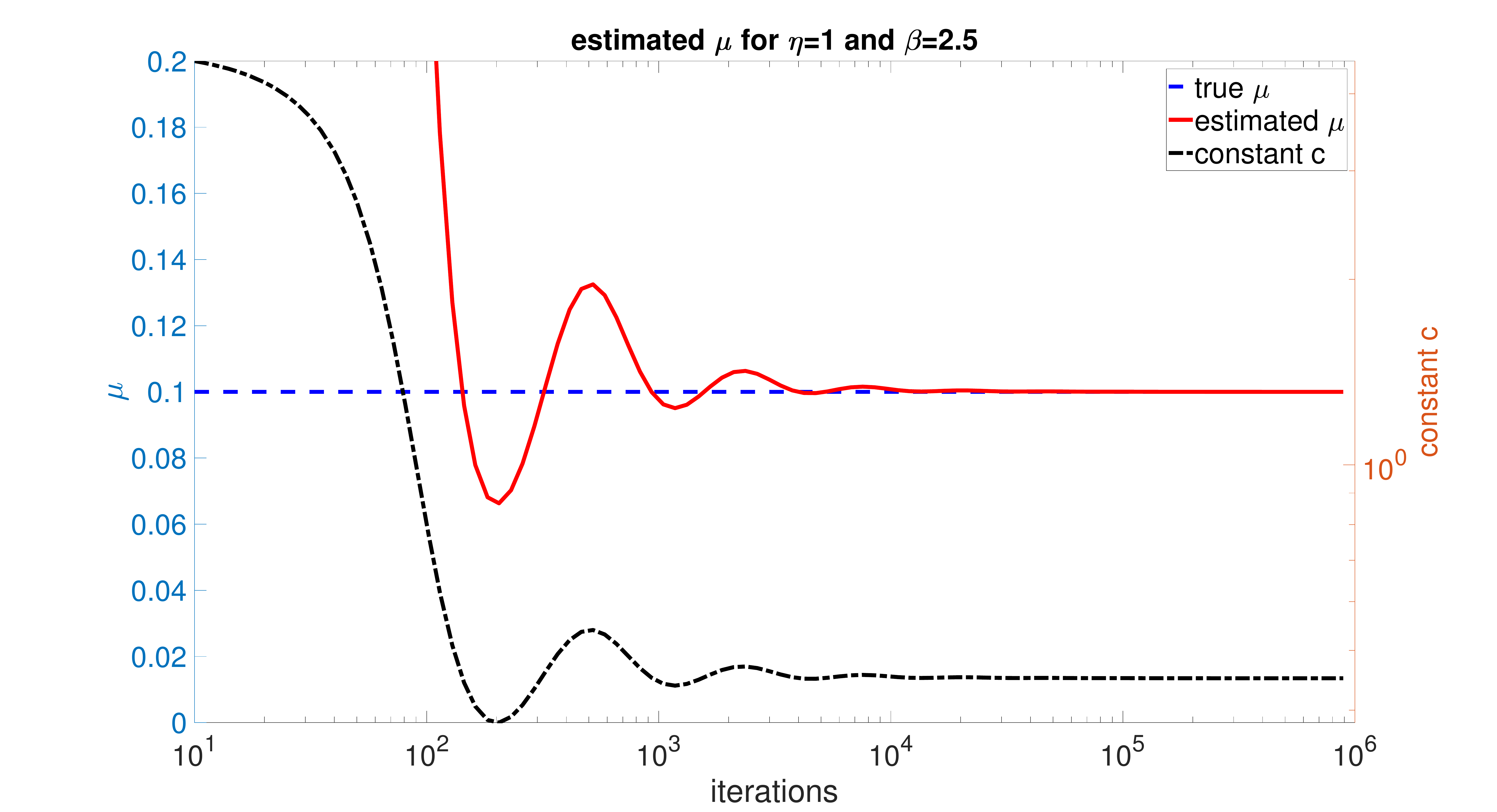}\caption{Demonstration of the method for $\eta=1$ and $\beta=2.5$. Dashed: true $\mu=0.1$, solid: estimated $\mu$, dash-dotted: estimated $c$; plotted over the iterations. From the stable section we estimate $\mu\approx0.1$.}\label{fig:diag1}
\end{figure}

\begin{figure}
\includegraphics[width=\textwidth]{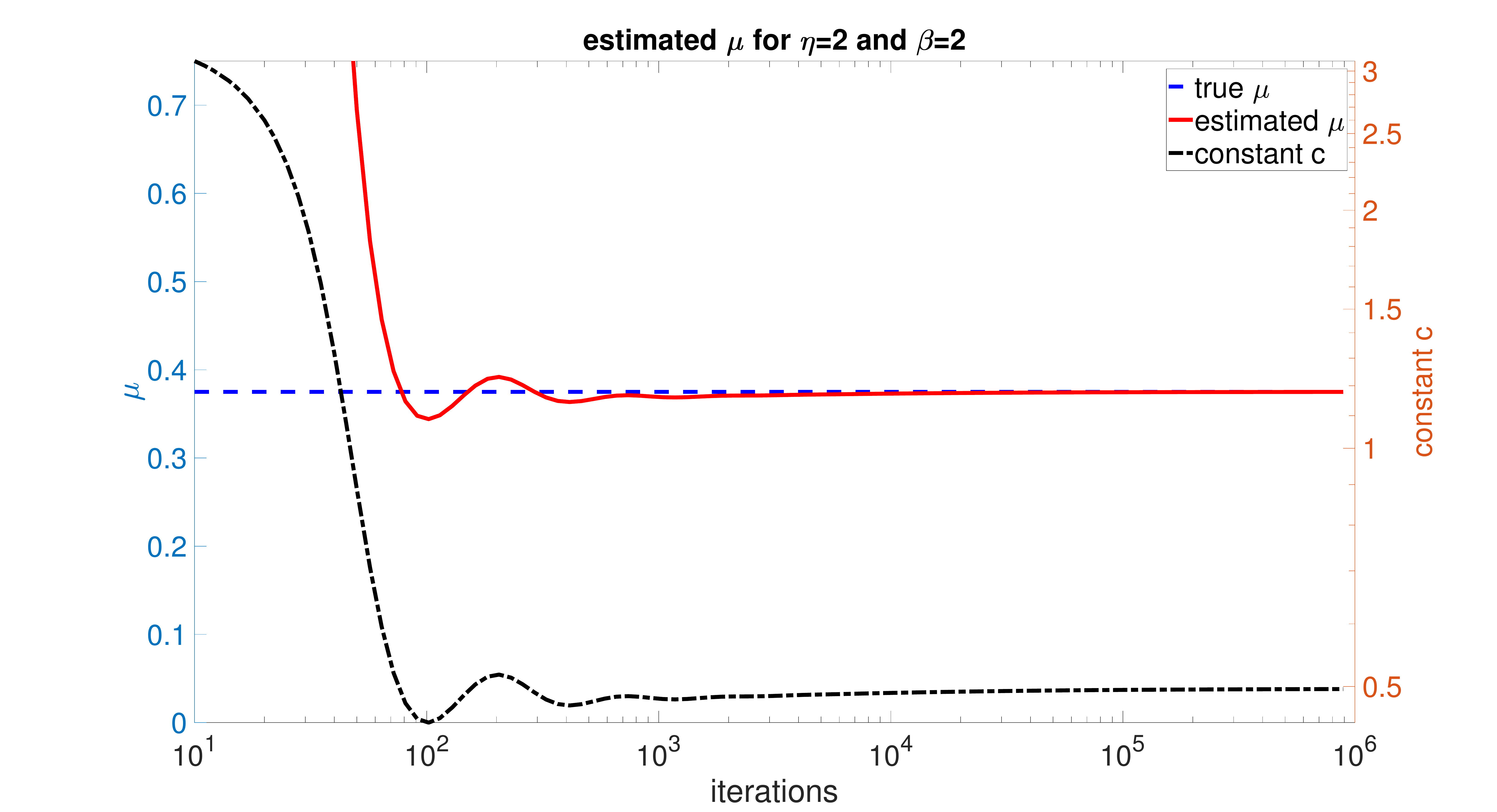}\caption{Demonstration of the method for $\eta=2$ and $\beta=2$. Dashed: true $\mu=0.375$, solid: estimated $\mu$, dash-dotted: estimated $c$; plotted over the iterations. From the stable section we estimate $\mu\approx 0.375$.}\label{fig:diag2}
\end{figure}

\begin{figure}
\includegraphics[width=\textwidth]{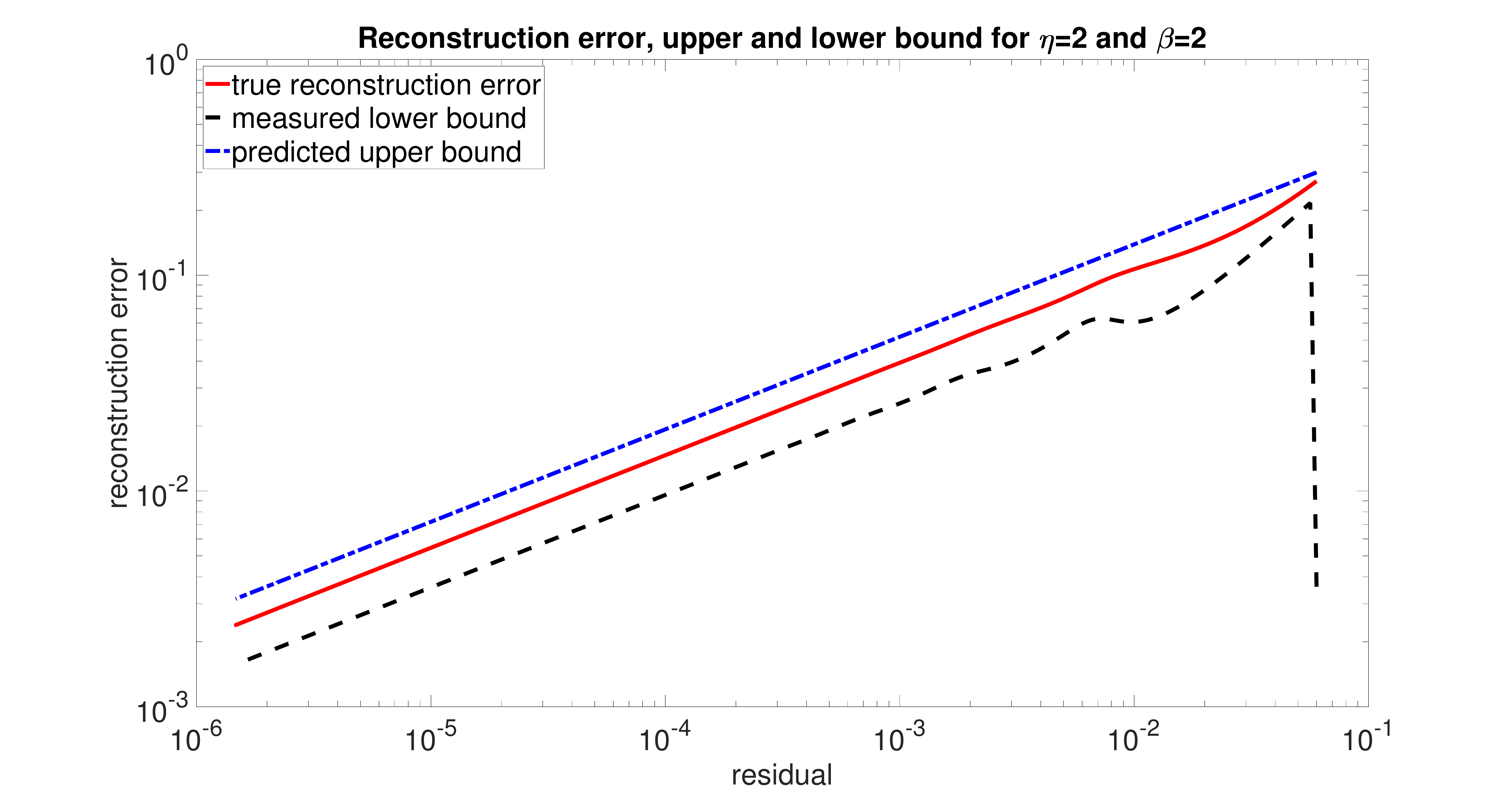}\caption{Reconstruction error for $\eta=2$ and $\beta=2$. Solid: measured error, dash-dotted: upper bound from the source condition, dashed: observed lower bound.}\label{fig:diag2LB}
\end{figure}

\begin{figure}
\includegraphics[width=\textwidth]{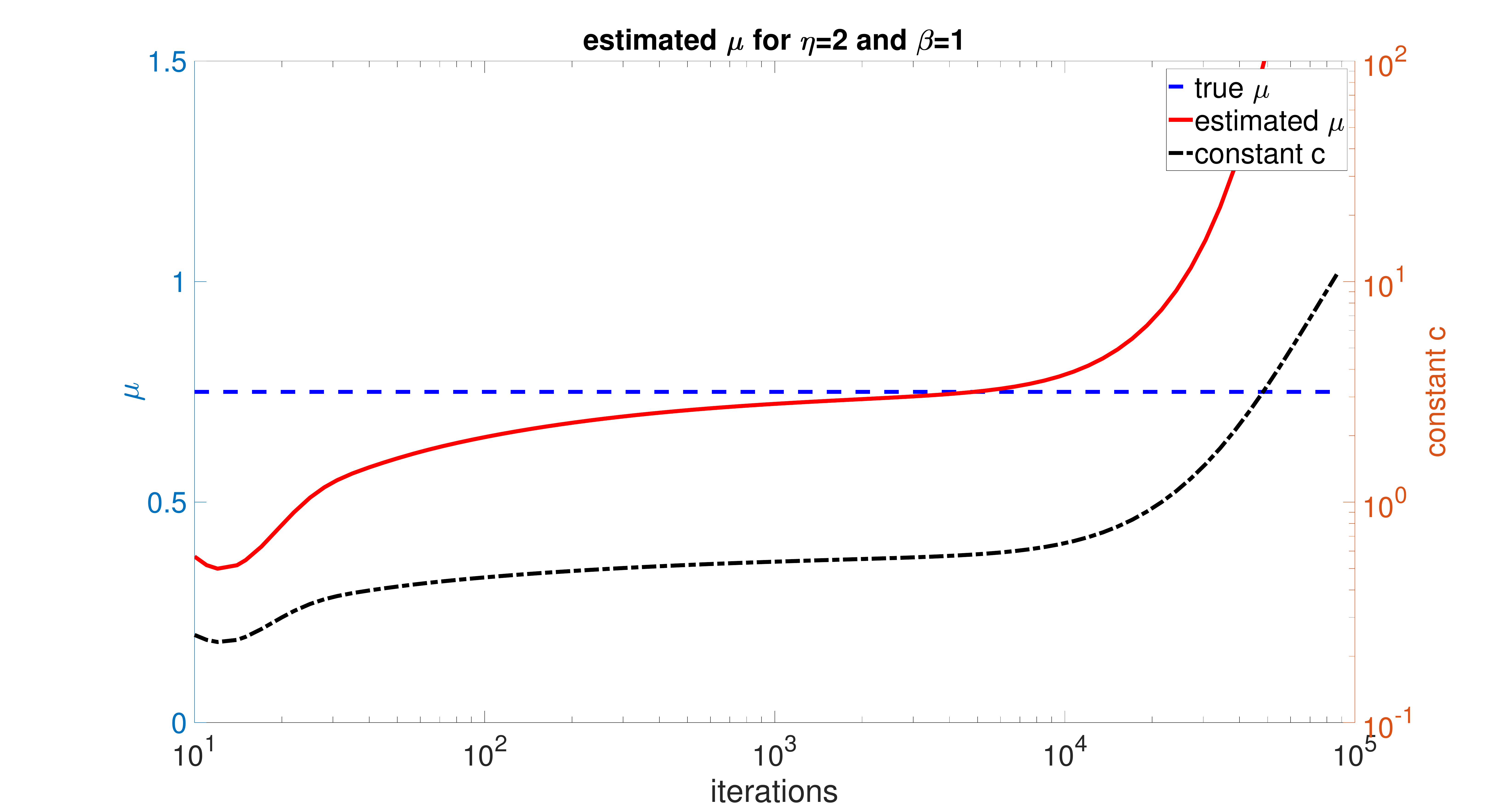}\caption{Demonstration of the method for $\eta=2$ and $\beta=1$. Dashed: true $\mu=0.75$, solid: estimated $\mu$, dash-dotted: estimated $c$; plotted over the iterations. From the stable section we estimate $\mu\approx0.7$. For higher iterations the estimate for $\mu$ goes up because our discretization level is insufficient.}\label{fig:diag3}
\end{figure}

\begin{figure}
\includegraphics[width=\textwidth]{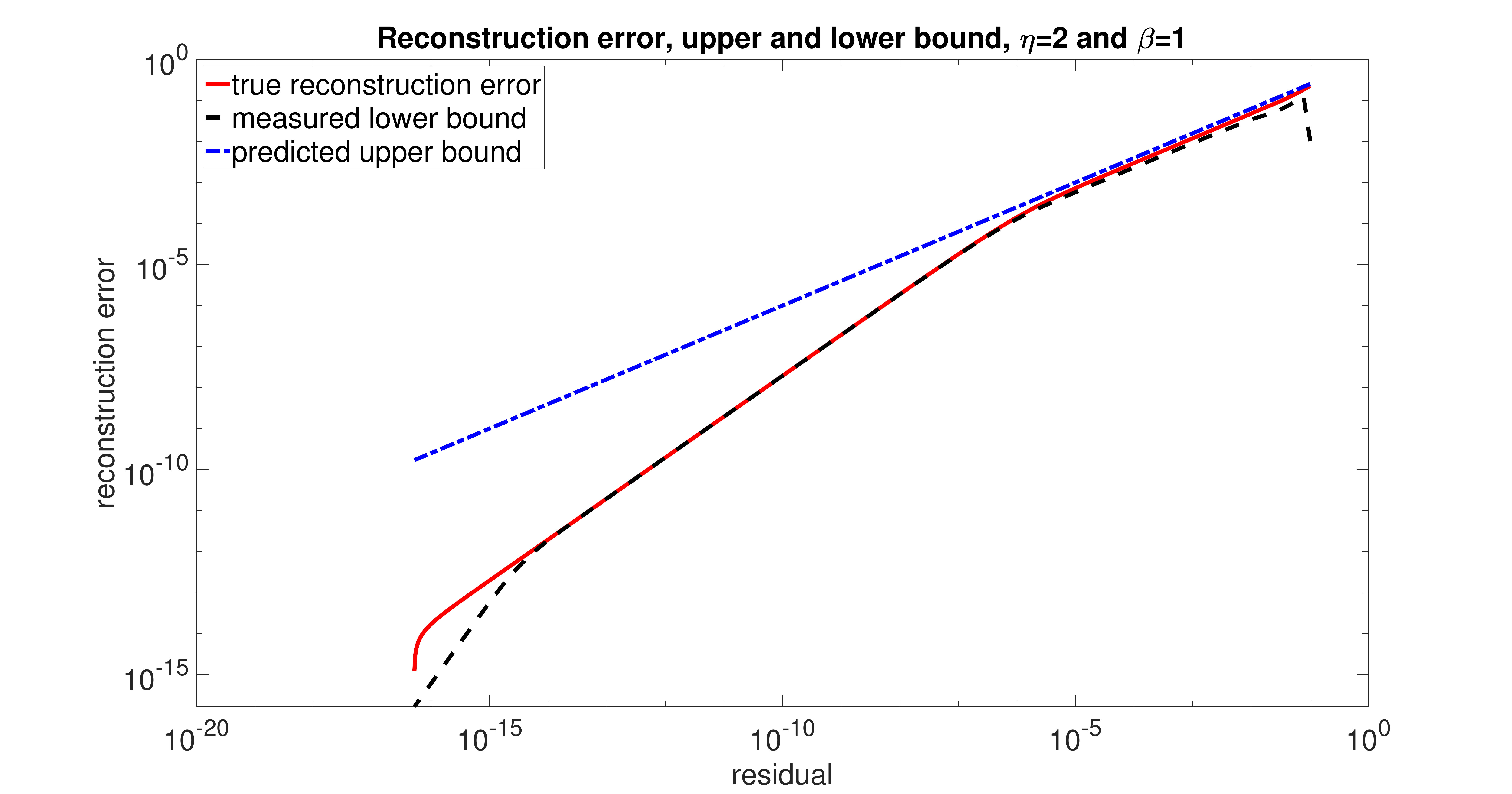}\caption{Reconstruction error for $\eta=2$ and $\beta=1$. Solid: measured error, dash-dotted: upper bound from the source condition, dashed: observed lower bound. The reconstruction error has two slopes. For larger residuals it corresponds to the correct $\mu$, for smaller residual the slope is approximately one due to insufficient discretization. }\label{fig:diag3LB}
\end{figure}

\begin{figure}
\includegraphics[width=\textwidth]{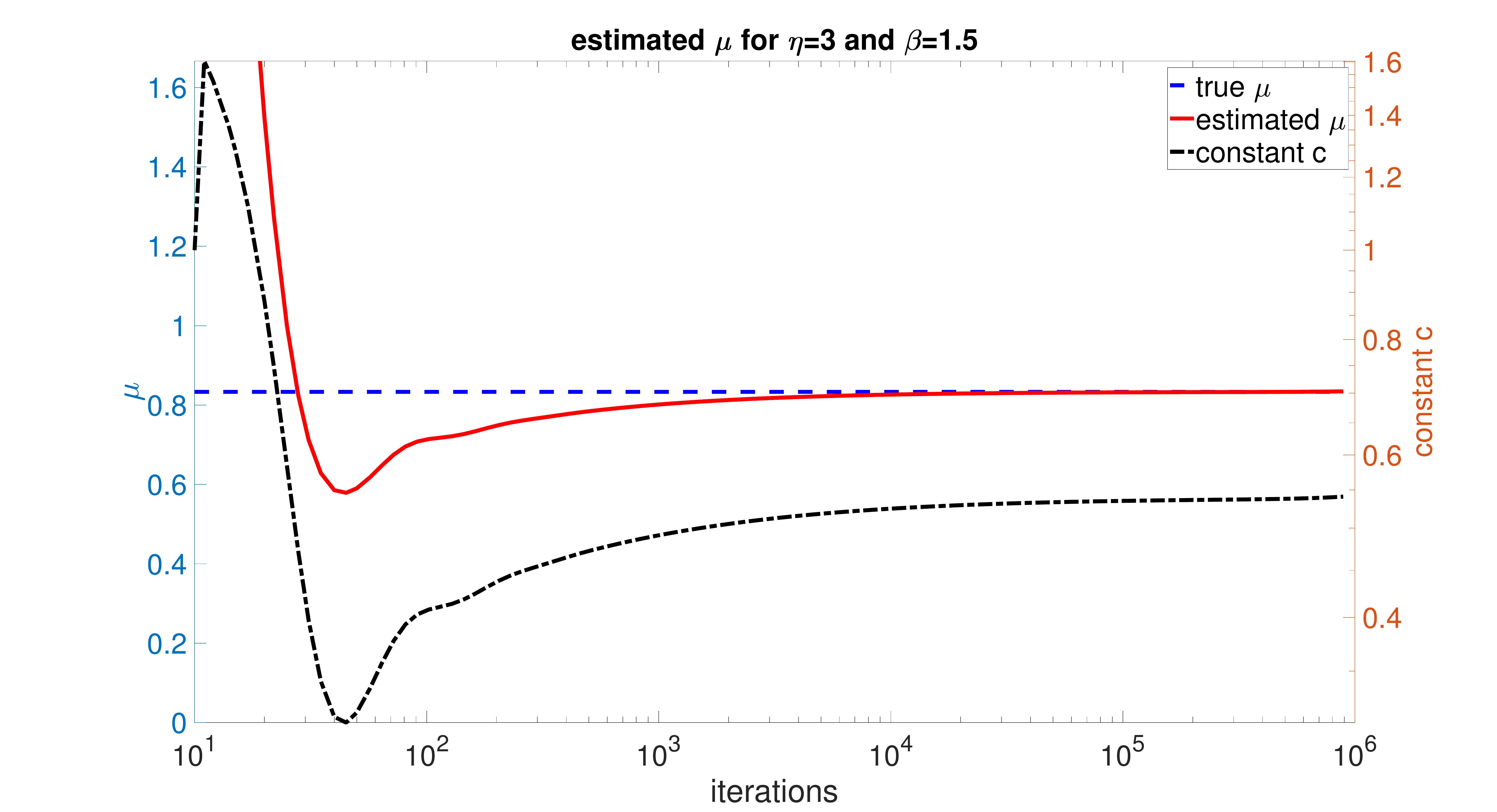}\caption{Demonstration of the method for $\eta=3$ and $\beta=1.5$. Dashed: true $\mu=0.833$, solid: estimated $\mu$, dash-dotted: estimated $c$; plotted over the iterations. From the stable section we estimate $\mu\approx 0.8$.}\label{fig:diag4}
\end{figure}

\begin{figure}
\includegraphics[width=\textwidth]{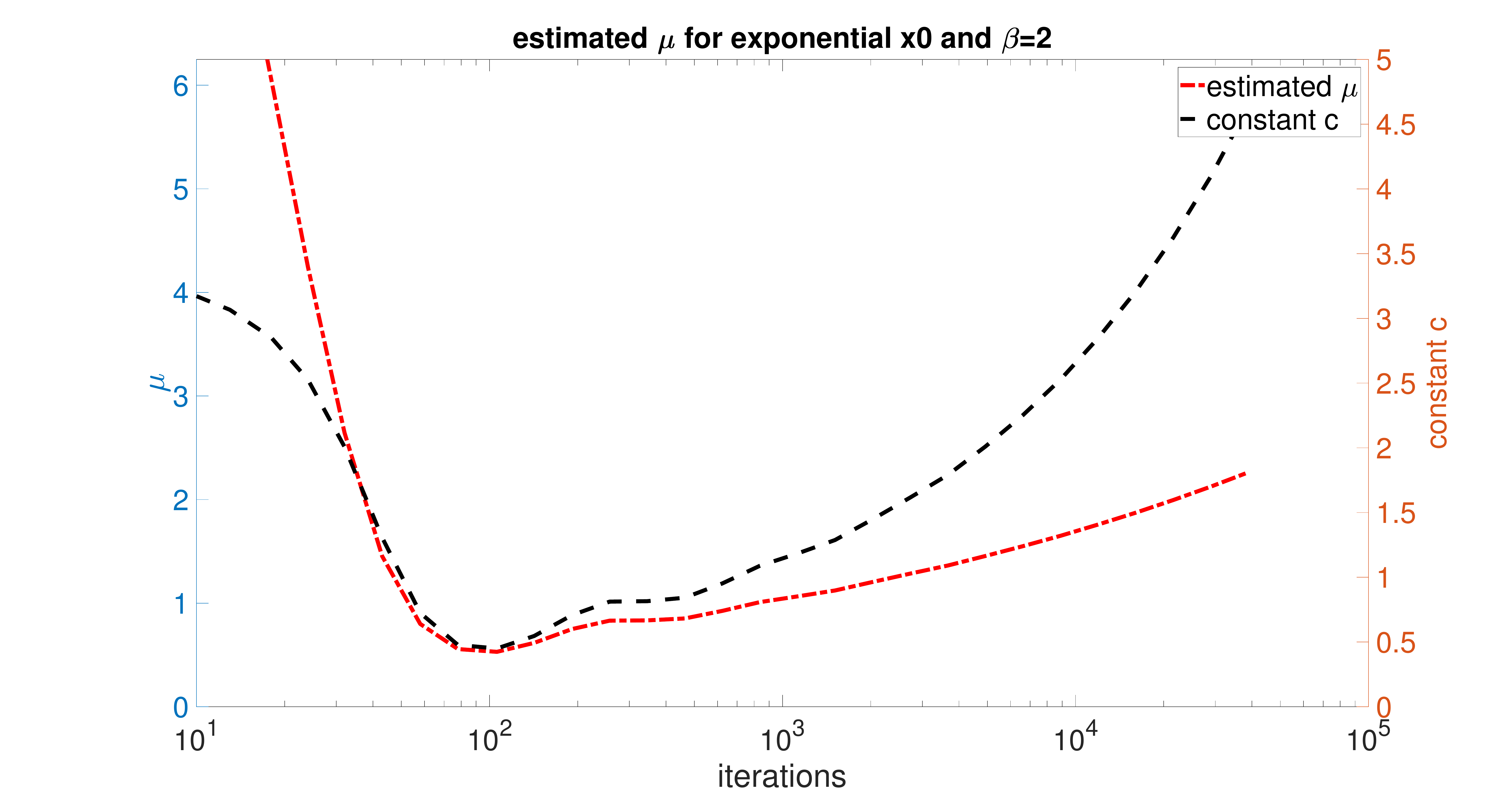}\caption{Demonstration of the method for $x_i=e^{-i}$ and $\beta=1.5$. Solid: estimated $\mu$, dash-dotted: estimated $c$; plotted over the iterations. Since both $\mu$ and $c$ are unstable we conclude that \eqref{eq:sc} is violated or not strictly fulfilled.}\label{fig:expon}
\end{figure}

\begin{figure}
\includegraphics[width=\textwidth]{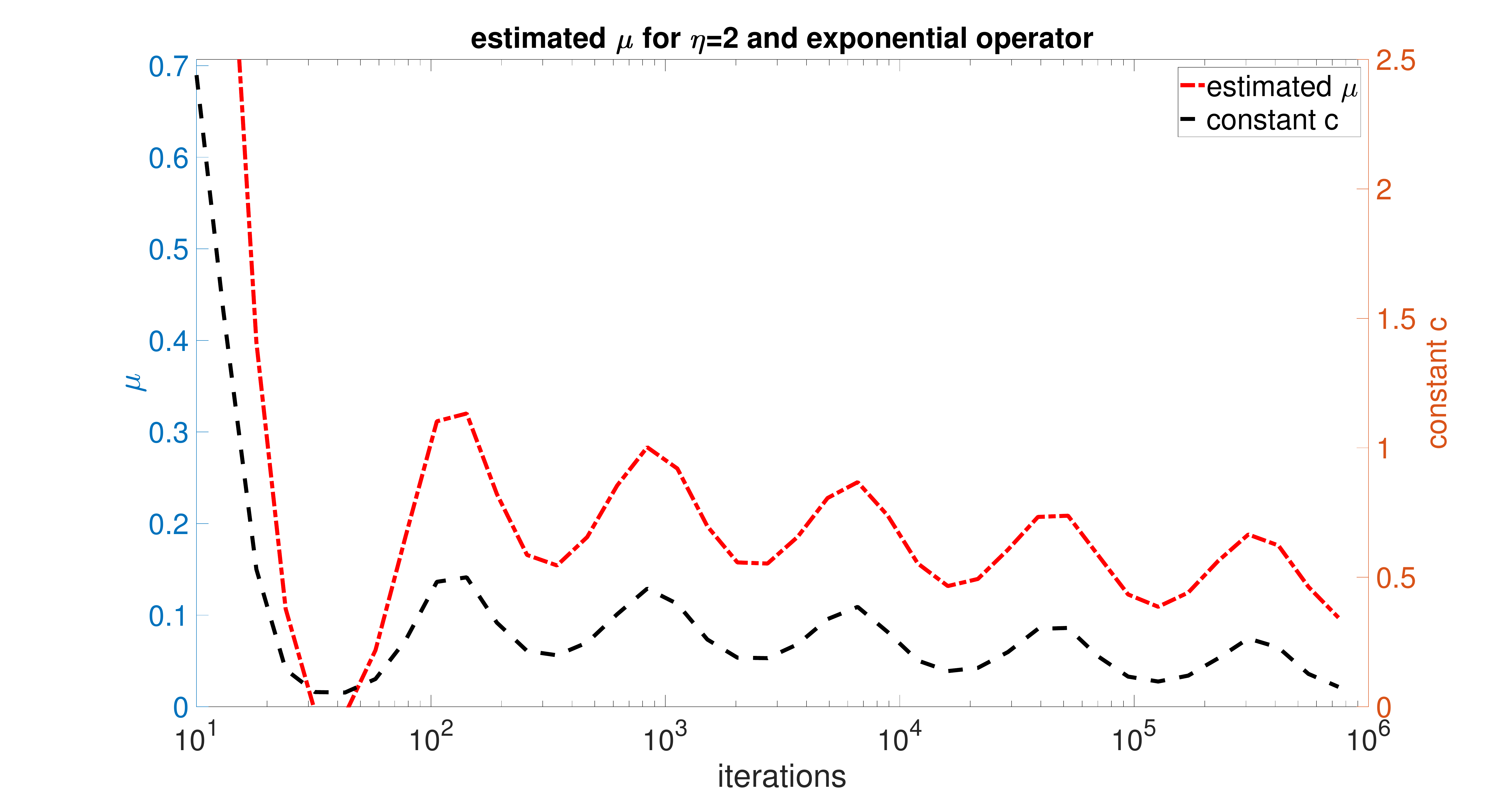}\caption{Demonstration of the method for $\eta=2$ and $\sigma_i=e^{-i}$. Solid: estimated $\mu$, dash-dotted: estimated $c$; plotted over the iterations. Since both $\mu$ and $c$ are unstable we conclude that \eqref{eq:sc} is violated.}\label{fig:expon_op}
\end{figure}

\begin{figure}
\includegraphics[width=\textwidth]{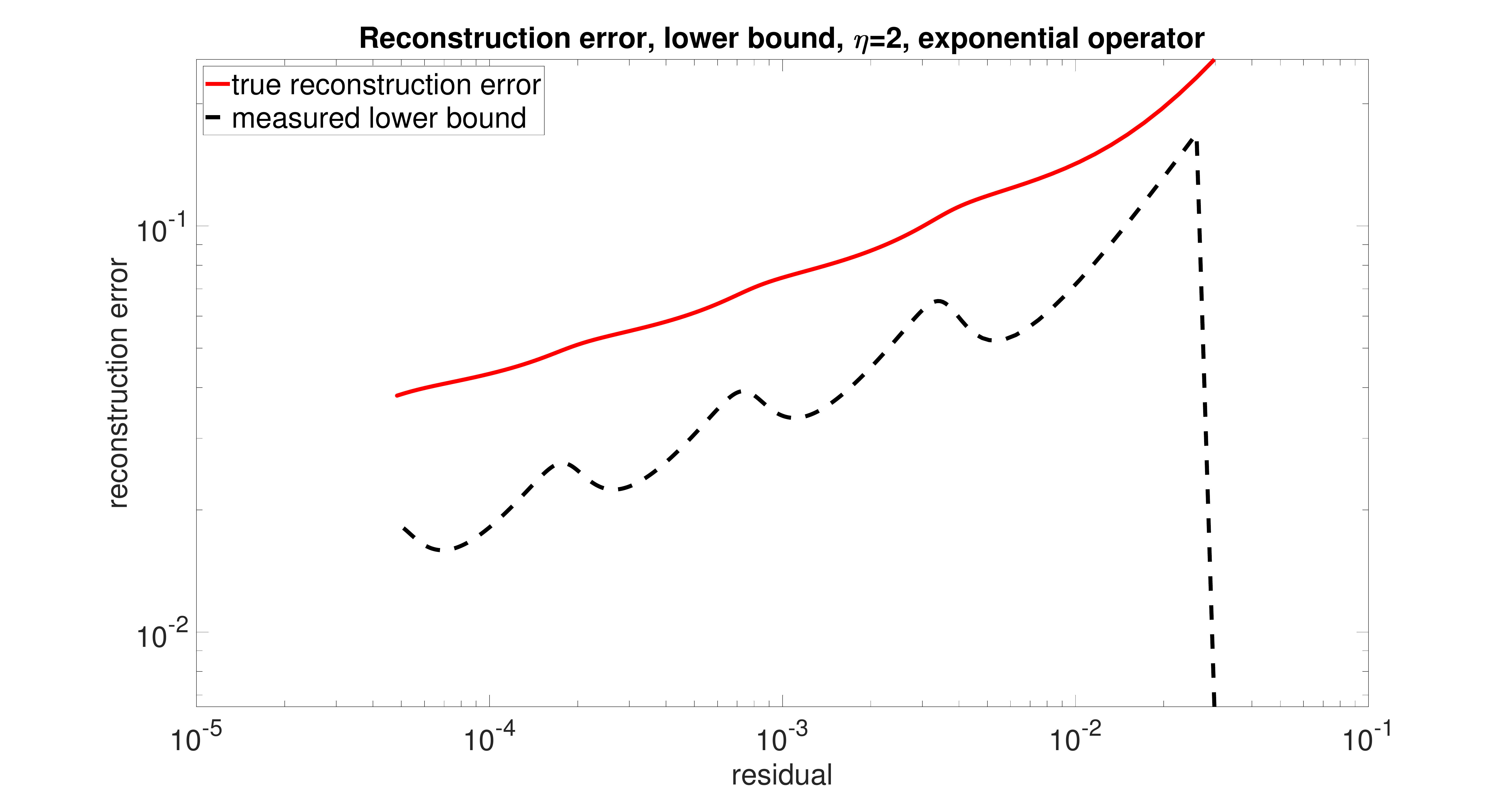}\caption{Reconstruction error for $\eta=2$ and $\sigma_i=e^{-i}$. Solid: measured error,  dashed: observed lower bound. The wildly oscillating lower bound suggests that the source condition \eqref{eq:sc2} does not hold.}\label{fig:expon_opLB}
\end{figure}

\subsection{Examples from Regularization Tools}
In the previous section we saw that the estimation technique yields reasonable results for the ``easy'' diagonal problem. Now we move to more realistic examples. To this end we use the \textit{Regularization Tools} toolbox \cite{regtools} in MATLAB. This toolbox comes with 14 problems of type \eqref{eq:problem}. In order to not be spoiled with background information we first apply our method to each of the problems. The algorithm yields seemingly good results for the problems \textit{deriv2} (see Figure \ref{fig:RT1}), \textit{gravity} (Figure \ref{fig:RT2}) and \textit{tomo}; see Figure \ref{fig:RT3}. At this point we remark that the routine setting up the \textit{tomo} data has a stochastic component, so all our experiments on the \textit{tomo} problem have been conducted with the same data set that we stored. Since we don't know $\mu$ exactly anymore, we must look for other ways of cross-checking it. We therefore try to verify the implication of Proposition \ref{thm:tikh}. Namely, if the estimate for $\mu$ is correct, the a priori parameter choice \eqref{eq:apriori} should yield the convergence rate \eqref{eq:rate_tikh}. In \eqref{eq:apriori} we simply set $c_1=c_2=1$. Since we are only interested in the exponent of the convergence rate this does not affect the result. We create noisy data contaminated with Gaussian error such that the relative error is between $10\%$ and $0.1\%$. Since we know the exact solution we can compute the reconstruction errors and a linear regression yields the observed convergence rate. The results are given in Table \ref{tab:res}. We achieve good accordance between the observed and the predicted convergence rate \eqref{eq:rate_tikh} for the \textit{deriv2} and \textit{tomo} problems. For \textit{gravity} we observe a mismatch between those rates. We therefore compute the singular value decomposition of all problems in the \textit{Regularization Tools} and look for the largest $\mu$ that still satisfies \eqref{eq:verif_smoothness}. It turns out that only the problems \textit{tomo} and \textit{deriv2} yield a reasonable result with $\mu \approx 0.1$ and $\mu\approx 0.2$, respectively. In particular we did not find a $\mu$ for the gravity problem. Looking again at Figure \ref{fig:RT2} and comparing it to Figure \ref{fig:expon_op} we see that in both Figures $\mu$ and $c$ show sinusoidal graphs. This may hint at an exponentially ill-posed operator in the gravity problem. 

In summary, we conclude that or method works and detects the two problems out of the Regularization Tools for which a source condition \eqref{eq:sc} seems to hold.

\begin{figure}
\includegraphics[width=\textwidth]{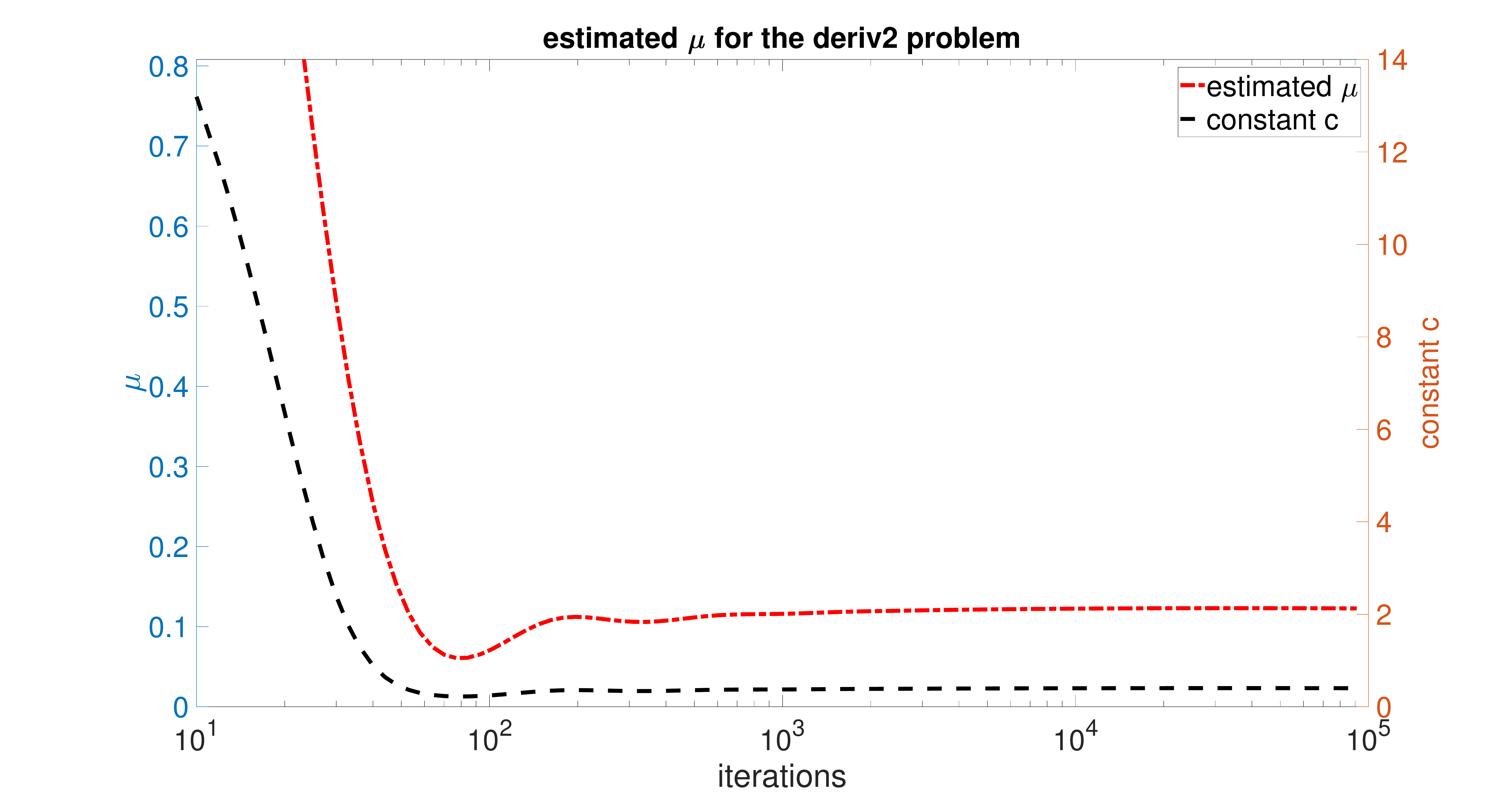}\caption{Demonstration of the method for the problem \textit{deriv2}. Solid: estimated $\mu$, dotted: estimated $c$; plotted over the iterations. From the stable section we estimate $\mu\approx 0.13$.}\label{fig:RT1}
\end{figure}

\begin{figure}
\includegraphics[width=\textwidth]{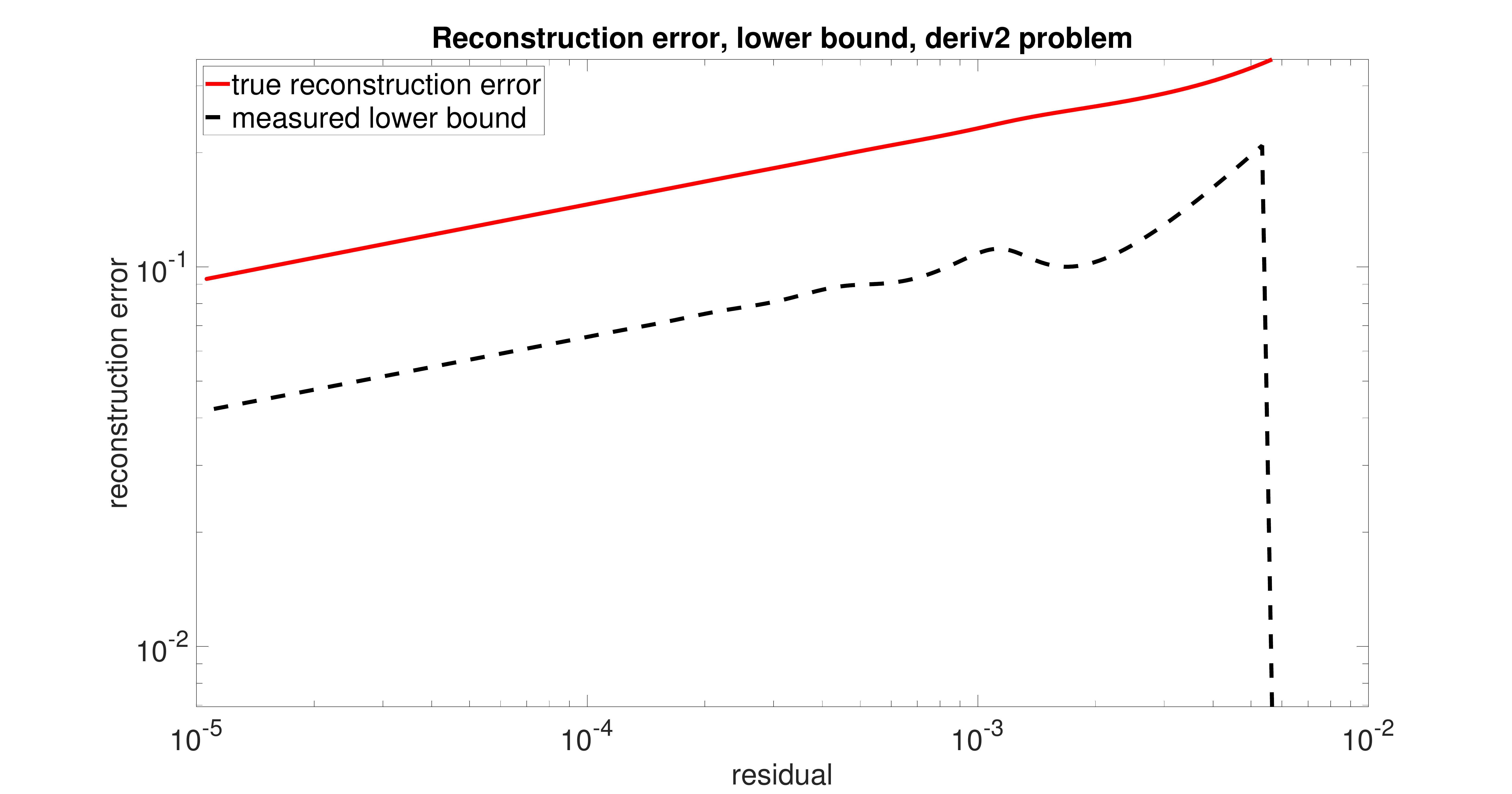}\caption{Reconstruction error for the problem \textit{deriv2}. Solid: measured error,  dashed: observed lower bound. The part with constant slope corresponds to a good estimation of $\mu$ in Figure \ref{fig:RT1}.}\label{fig:RT1LB}
\end{figure}

\begin{figure}
\includegraphics[width=\textwidth]{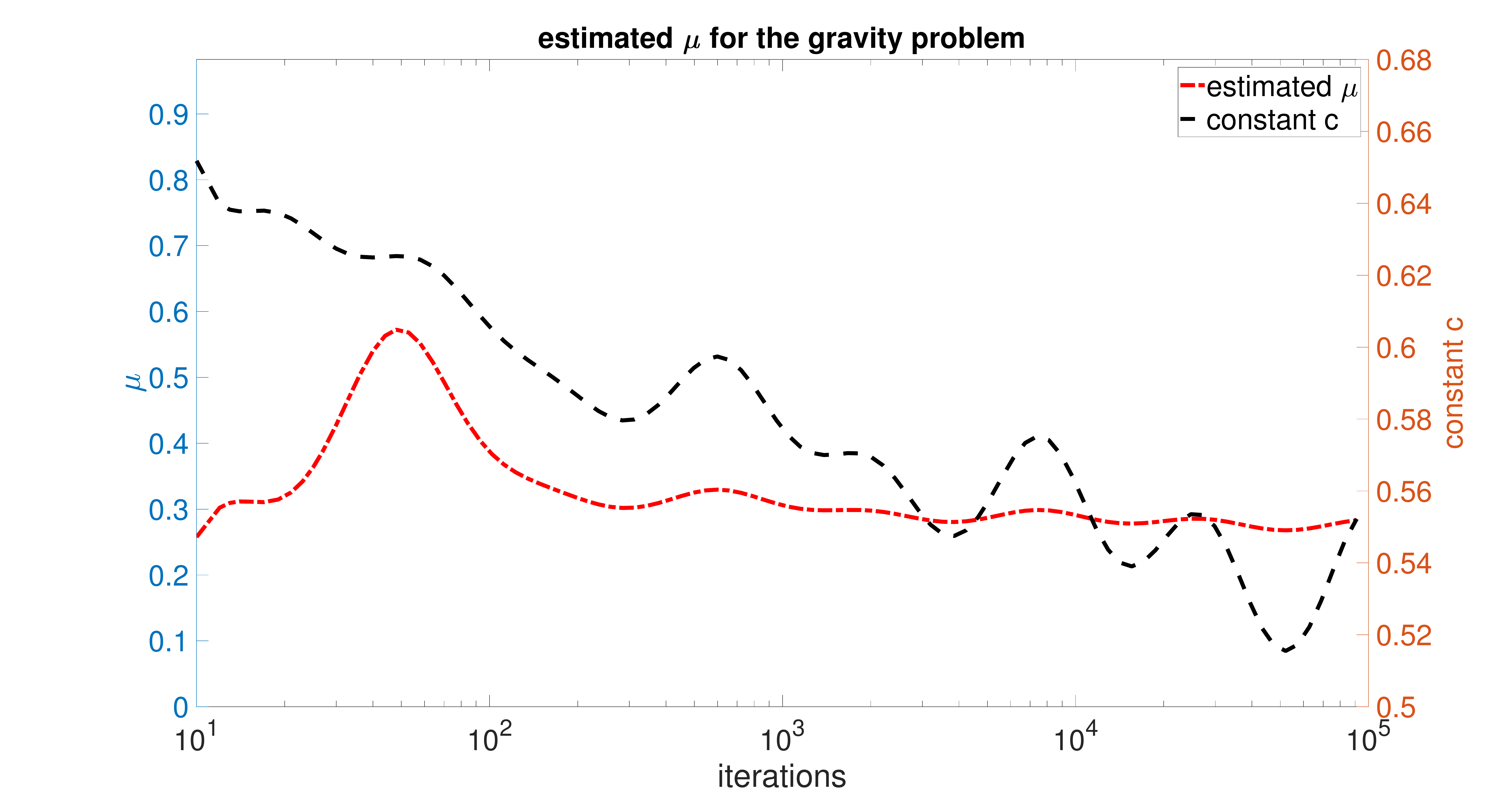}\caption{Demonstration of the method for the problem \textit{gravity}. Solid: estimated $\mu$, dotted: estimated $c$; plotted over the iterations. From the stable section we estimate $\mu\approx 0.3$.}\label{fig:RT2}
\end{figure}

\begin{figure}
\includegraphics[width=\textwidth]{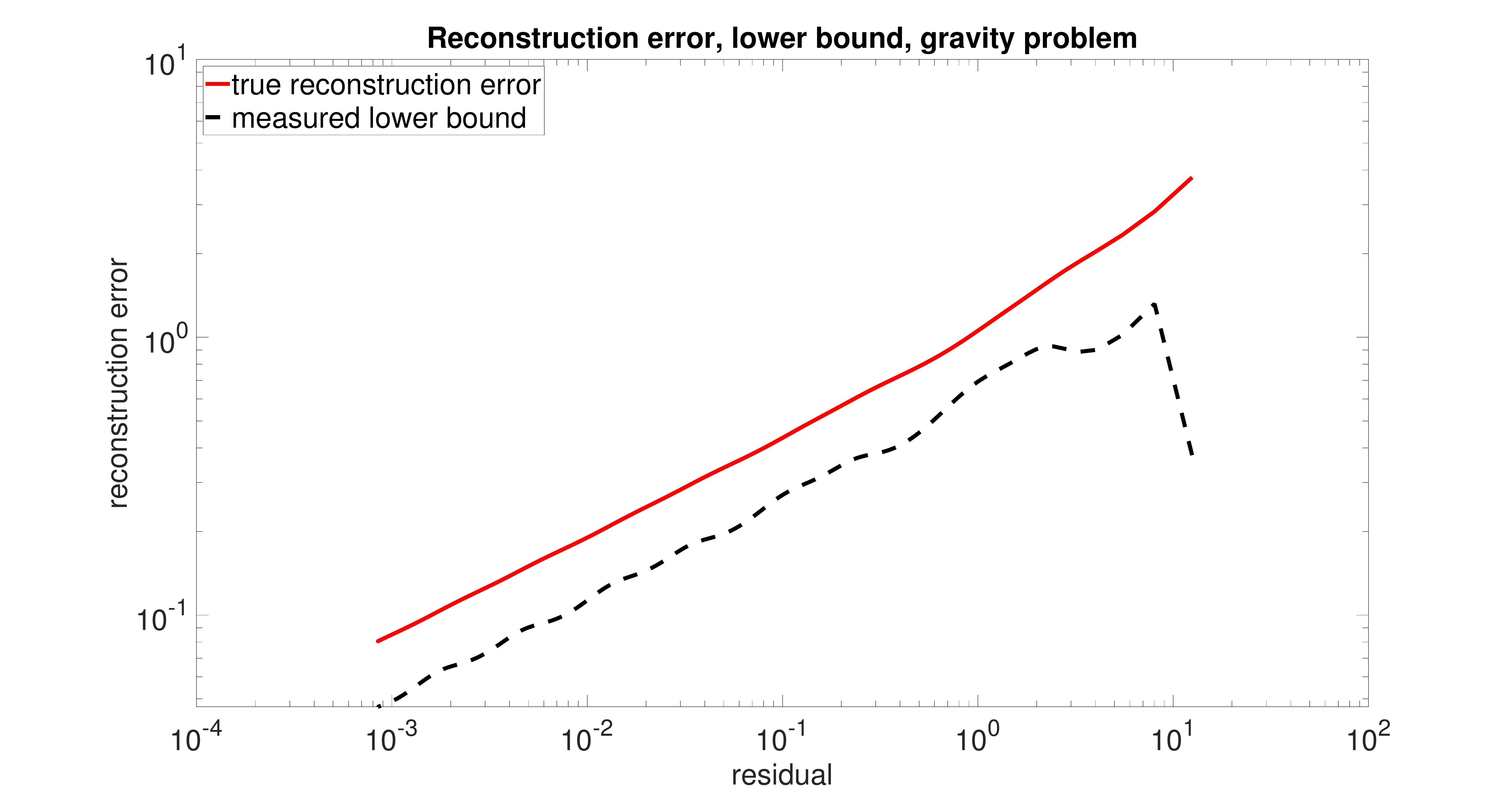}\caption{Reconstruction error for the problem \textit{gravity}. Solid: measured error,  dashed: observed lower bound. The lower bound is oscillating similar as in Figure \ref{fig:expon_opLB}, albeit with smaller amplitude.}\label{fig:RT2LB}
\end{figure}

\begin{figure}
\includegraphics[width=\textwidth]{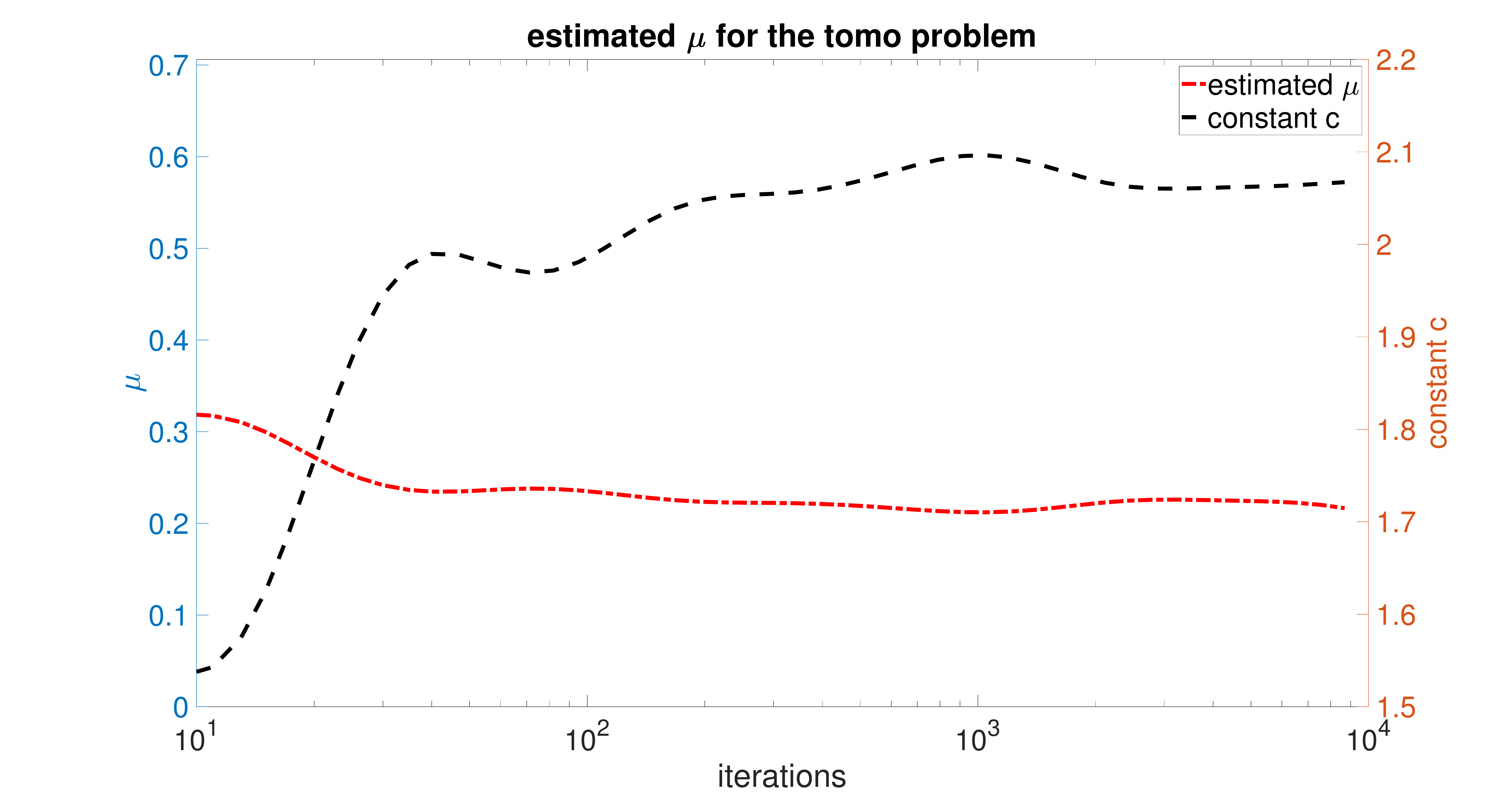}\caption{Demonstration of the method for the problem \textit{tomo}. Solid: estimated $\mu$, dotted: estimated $c$; plotted over the iterations. From the stable section we estimate $\mu\approx 0.2$.}\label{fig:RT3}
\end{figure}

\begin{figure}
\includegraphics[width=\textwidth]{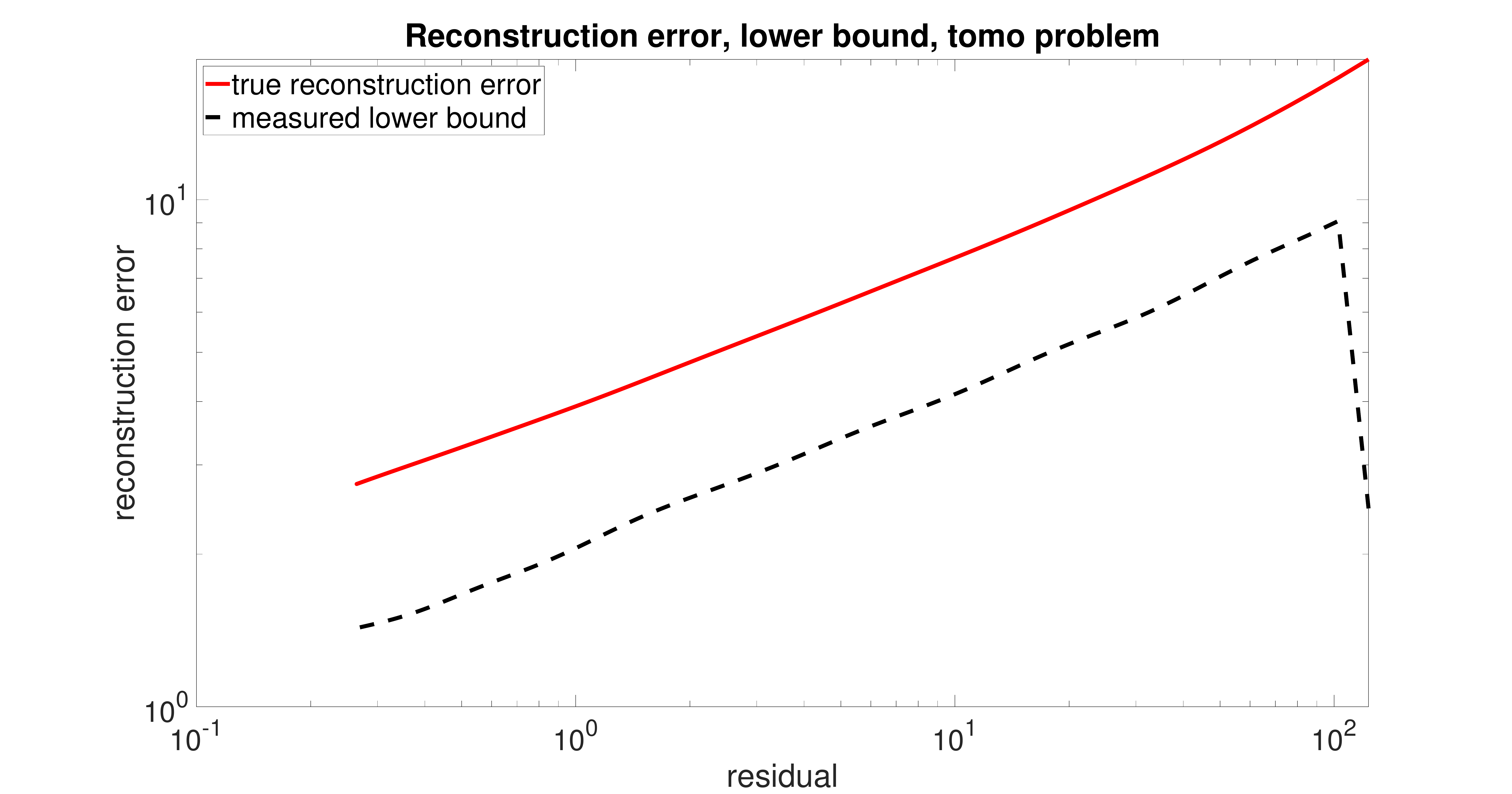}\caption{Reconstruction error for the problem \textit{tomo}. Solid: measured error,  dashed: observed lower bound. }\label{fig:RT3LB}
\end{figure}

\begin{table}\centering \caption{Tikhonov-test for the problems \textit{deriv2}, \textit{tomo}, and \textit{gravity}. For the first two problems the predicted and observed convergence rates are close enough to suggest that $\mu$ is estimated reasonably well. The \textit{gravity} problem yields a misfit since it does not fulfill \eqref{eq:verif_smoothness} for any $\mu>0$.}
\begin{tabular}{c c c c}
problem & estimated $\mu$ & predicted rate & observed rate \\
\textit{deriv2} & $0.13$ & $\delta^{0.206}$ & $\delta^{0.193}$ \\
\textit{tomo} & $0.2$ & $\delta^{0.285}$ & $\delta^{0.3}$ \\
\textit{gravity} & $0.3$ & $\delta^{0.375}$ & $\delta^{0.527}$
\end{tabular}\label{tab:res}
\end{table}

\subsection{Noisy data}
So far our experiments for the estimation of $\mu$ involved only noise-free data. In practice of course noise is inevitable. We repeat the tests of Section \ref{sec:diag} for noisy data. Since the results are similar we will only discuss the case $\eta=2$ and $\beta=2$ as an example. We show the results for $1\%$ and $0.1\%$ relative noise in Figure \ref{fig:noise1} and Figure \ref{fig:noise2}, respectively. We observe that in both cases there is a part where $c$ and $\mu$ are approximately stable. The estimate for $\mu$ from this section is slightly smaller than the true value. As the number of iterations grows, the estimated $\mu$ drops even below zero while at the same point $c$ increases strongly. If we increase the noise level further we obtain a smaller and smaller region for $\mu$. Nevertheless, at least for smaller noise levels the method still works. The pattern also holds for the examples from Regularization Tools. We provide a test of the \textit{deriv2}-problem with $0.1\% $ relative noise in Figure \ref{fig:noise_deriv}. For the \textit{tomo} problem, we still obtain a reasonable estimate for $1\%$ noise, see Figure \ref{fig:noise_tomo}. As with all noisy data sets, the estimate for $\mu$ starts to decrease rapidly after some point. Looking at the lower bound, e.g. in Figure \ref{fig:noise_tomo_LB} for the \textit{tomo} problem, we see that for small er residuals the lower bound starts to increase with smaller residuals. We conclude that then the noise dominates the KL inequality and the estimates obtained for $\mu$ in the corresponding later iterations of the Landweber method can be disregarded.

\subsection{Real measurements}
Motivated by this, we now turn to the last test case where we apply our algorithm to real data. This means we have an unknown amount of noise and no exact solution to repeat the Tikhonov-experiment. We use the tomographic X-ray data of a carved cheese, a lotus root, and of a walnut which are freely available at \url{http://www.fips.fi/dataset.php}; see also the documentations \cite{cheese},\cite{lotus}, and \cite{nuts}. We use the datasets \textit{DataFull128x15.mat}, \textit{LotusData128.mat}, and \textit{Data82.mat}, respectively. Since the matrices $A$ are far too large for a full SVD we can not verify \eqref{eq:verif_smoothness} and we have to rely solely on our new approach with the Landweber method which is still easily computable. The results are shown in Figures \ref{fig:cheese}, \ref{fig:lotus}, and \ref{fig:nuts}. The three results look similar but vary in detail. In all figures we have a maximum of $\mu$ after around 25 iterations. After that, $\mu$ decreases while the constant $c$ increases. This is likely due to the noise which takes over at some point see the explanation at the end of the last section. While in particular the lotus problem yields an almost stable section for $\mu$, the walnut hardly yields a trustworthy region. This suggests that the solution of the walnut problem is furthest from fulfilling the source condition \eqref{eq:sc}. This can also be concluded from Figure \ref{fig:real_lower}, where we plot the measured lower bounds from \eqref{eq:lb}.

\begin{figure}
\includegraphics[width=\textwidth]{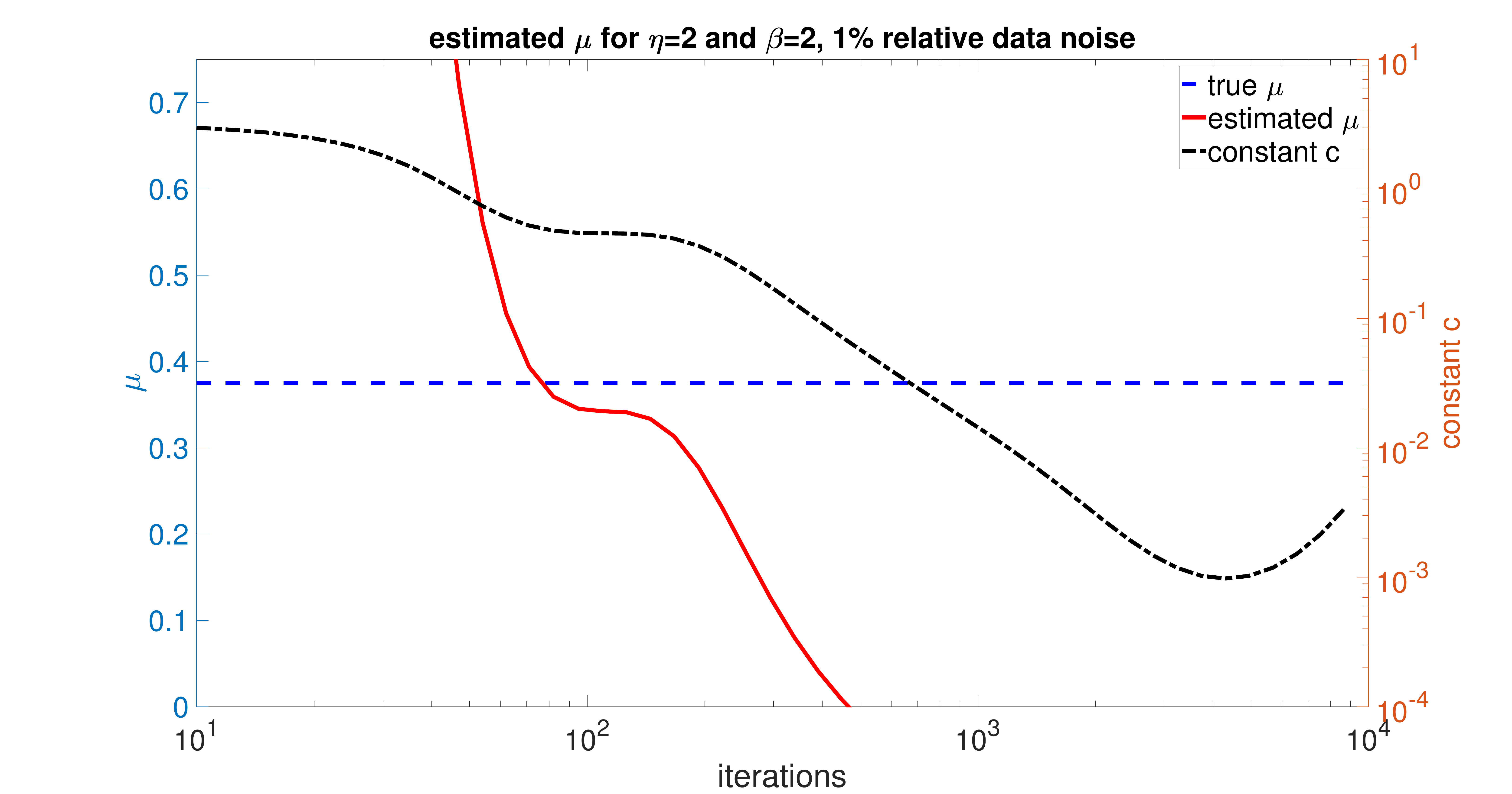}\caption{Demonstration of the method for $\eta=2$ and $\beta=2$ with $1\%$ relative data noise. Solid: true $\mu=0.375$, dashed: estimated $\mu$, dotted: estimated $c$; plotted over the iterations. In the approximately stable region we estimate $\mu\approx 0.35$.}\label{fig:noise1}
\end{figure}

\begin{figure}
\includegraphics[width=\textwidth]{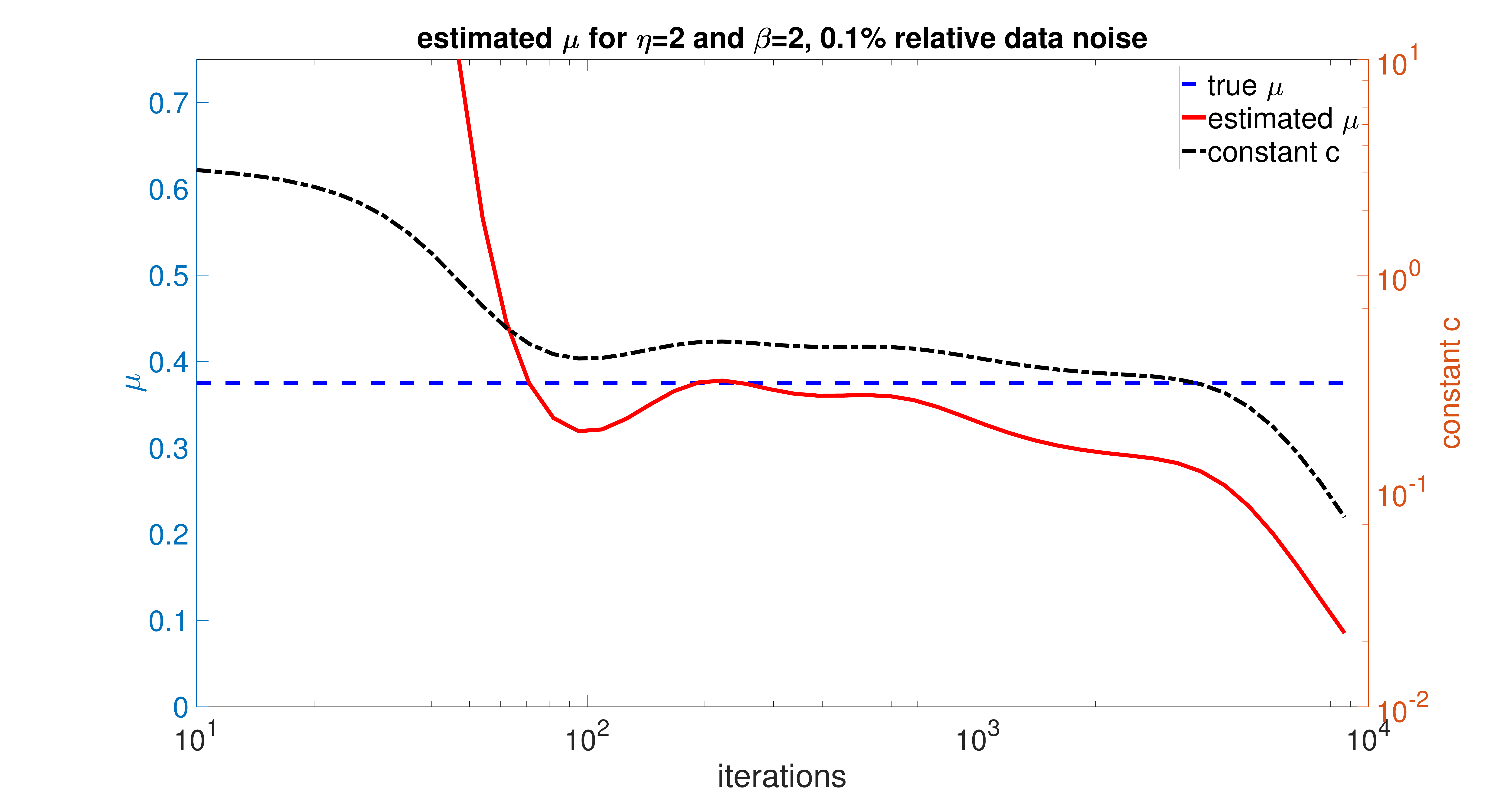}\caption{Demonstration of the method for $\eta=2$ and $\beta=2$ with $0.1\%$ relative data noise. Solid: true $\mu=0.375$, dashed: estimated $\mu$, dotted: estimated $c$; plotted over the iterations. In the approximately stable region we estimate $\mu\approx 0.4$.}\label{fig:noise2}
\end{figure}
\begin{figure}
\includegraphics[width=\textwidth]{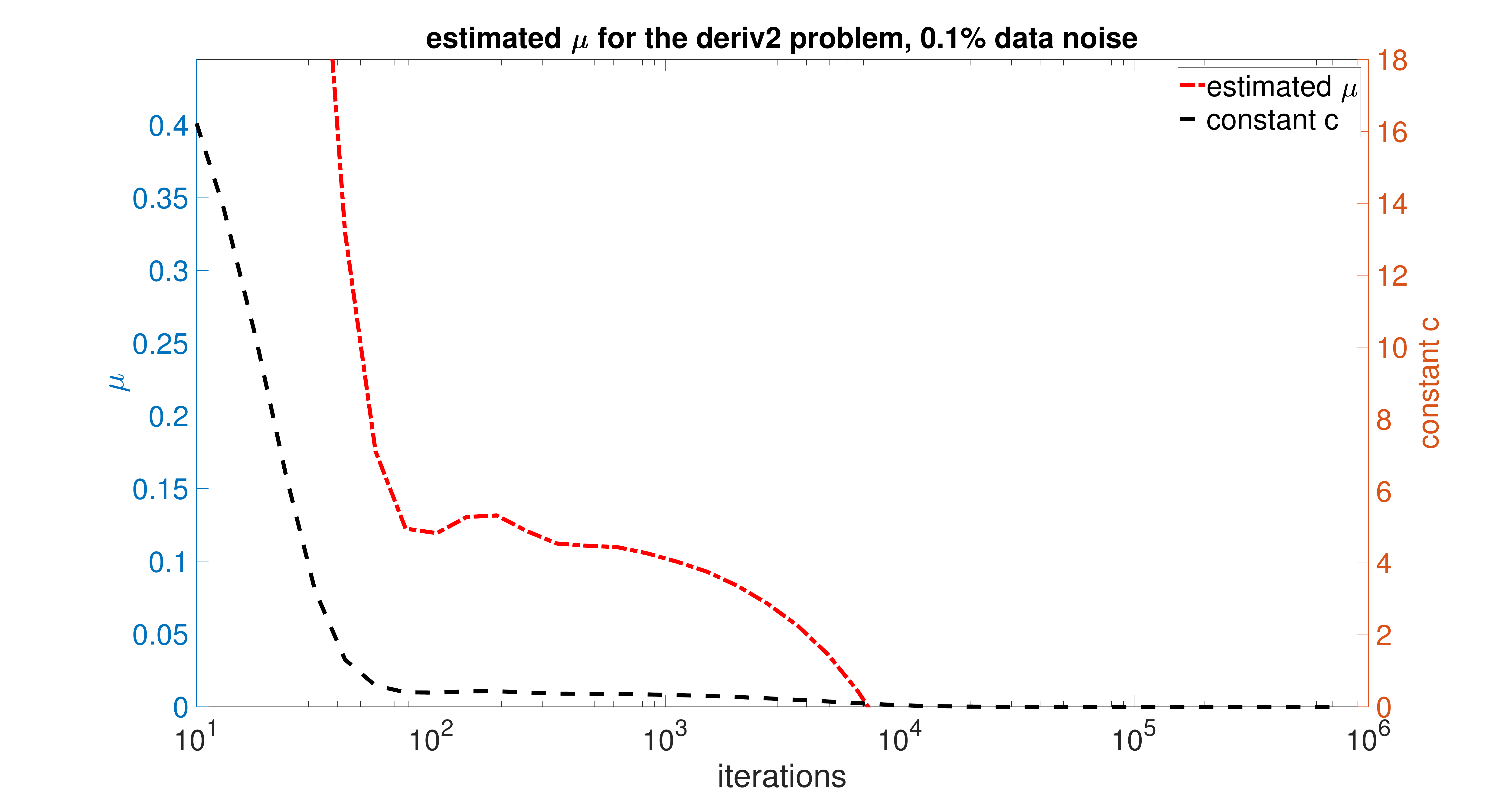}\caption{Demonstration of the method for the problem \textit{deriv2} with $0.1\%$ relative data noise. Dashed: estimated $\mu$, dotted: estimated $c$; plotted over the iterations. In the approximately stable region we may take $\mu\approx 0.13$.}\label{fig:noise_deriv}
\end{figure}

\begin{figure}
\includegraphics[width=\textwidth]{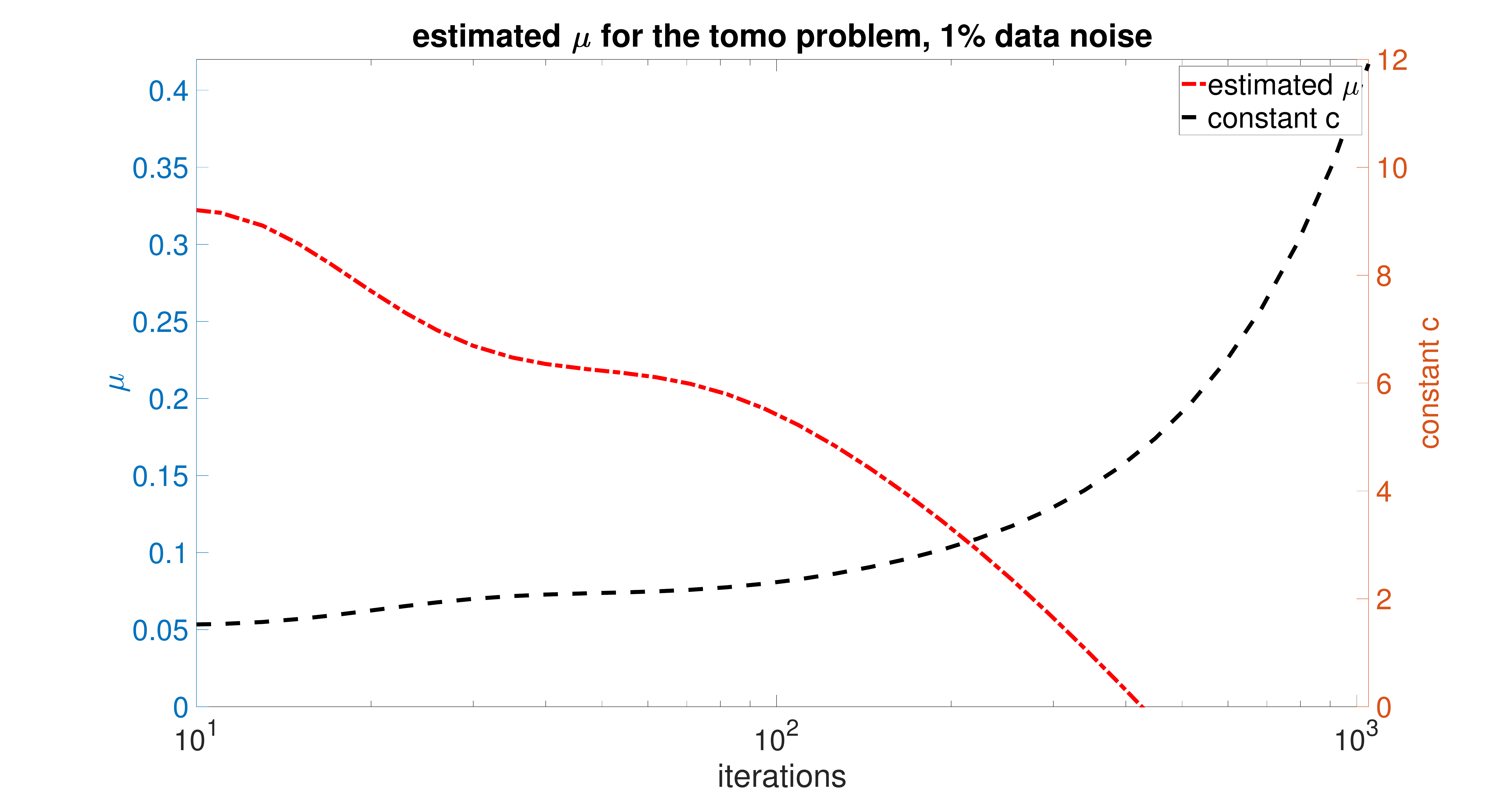}\caption{Demonstration of the method for the problem \textit{tomo} with $1\%$ relative data noise. Dashed: estimated $\mu$, dotted: estimated $c$; plotted over the iterations. In the saddlepoint-like region we may take $\mu\approx 0.22$. Afterwards $\mu$ decreases because the noise becomes dominating, see Figure \ref{fig:noise_tomo_LB}.}\label{fig:noise_tomo}
\end{figure}

\begin{figure}
\includegraphics[width=\textwidth]{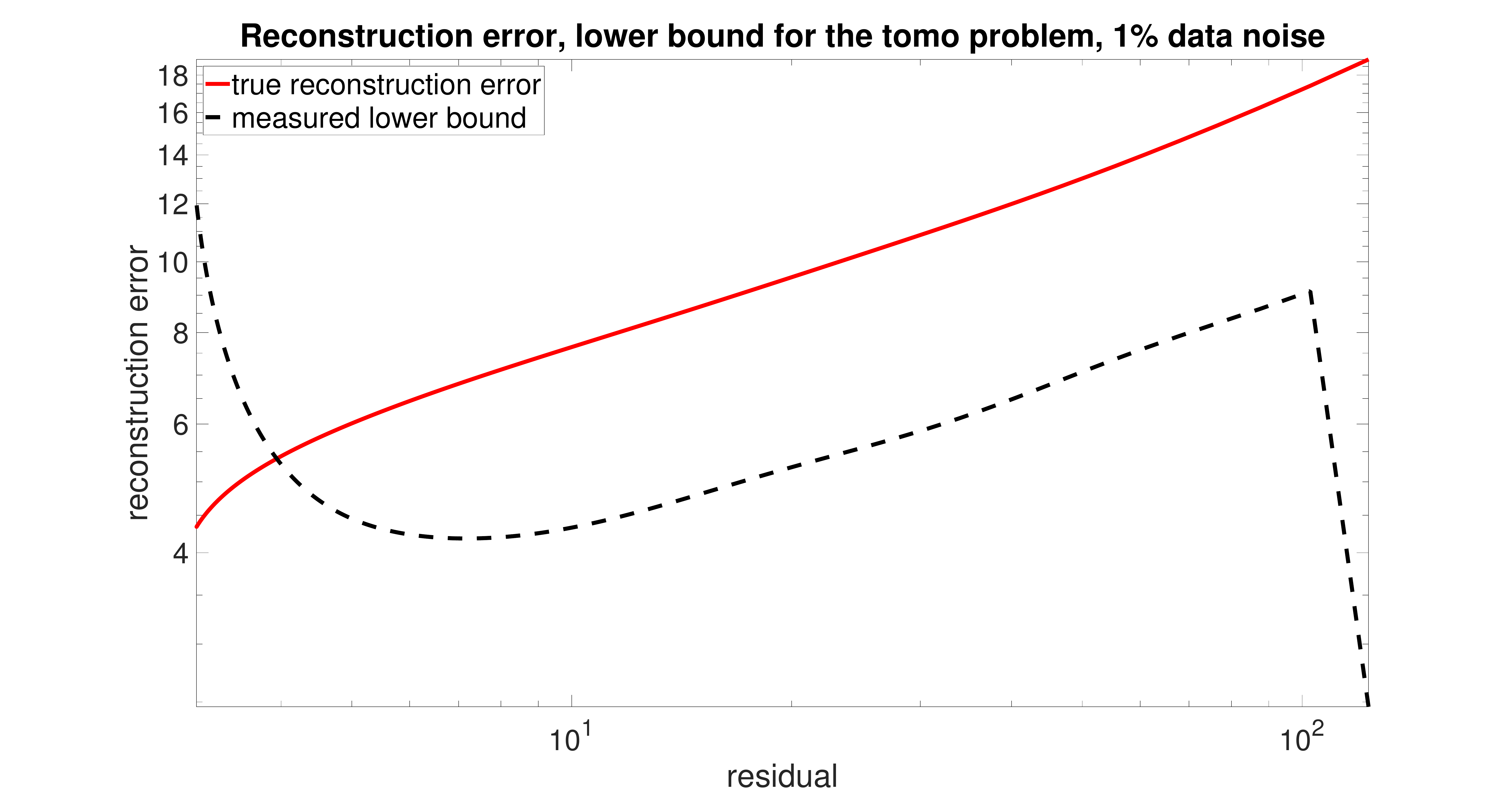}\caption{Reconstruction error for the problem \textit{tomo} with $1\%$ relative data noise. Solid: measured error,  dashed: observed lower bound. The lower bound increasing for lower residuals means that the noise becomes dominant, hence $\mu$ is not estimated correctly in in those regions.}\label{fig:noise_tomo_LB}
\end{figure}

\begin{figure}
\includegraphics[width=\textwidth]{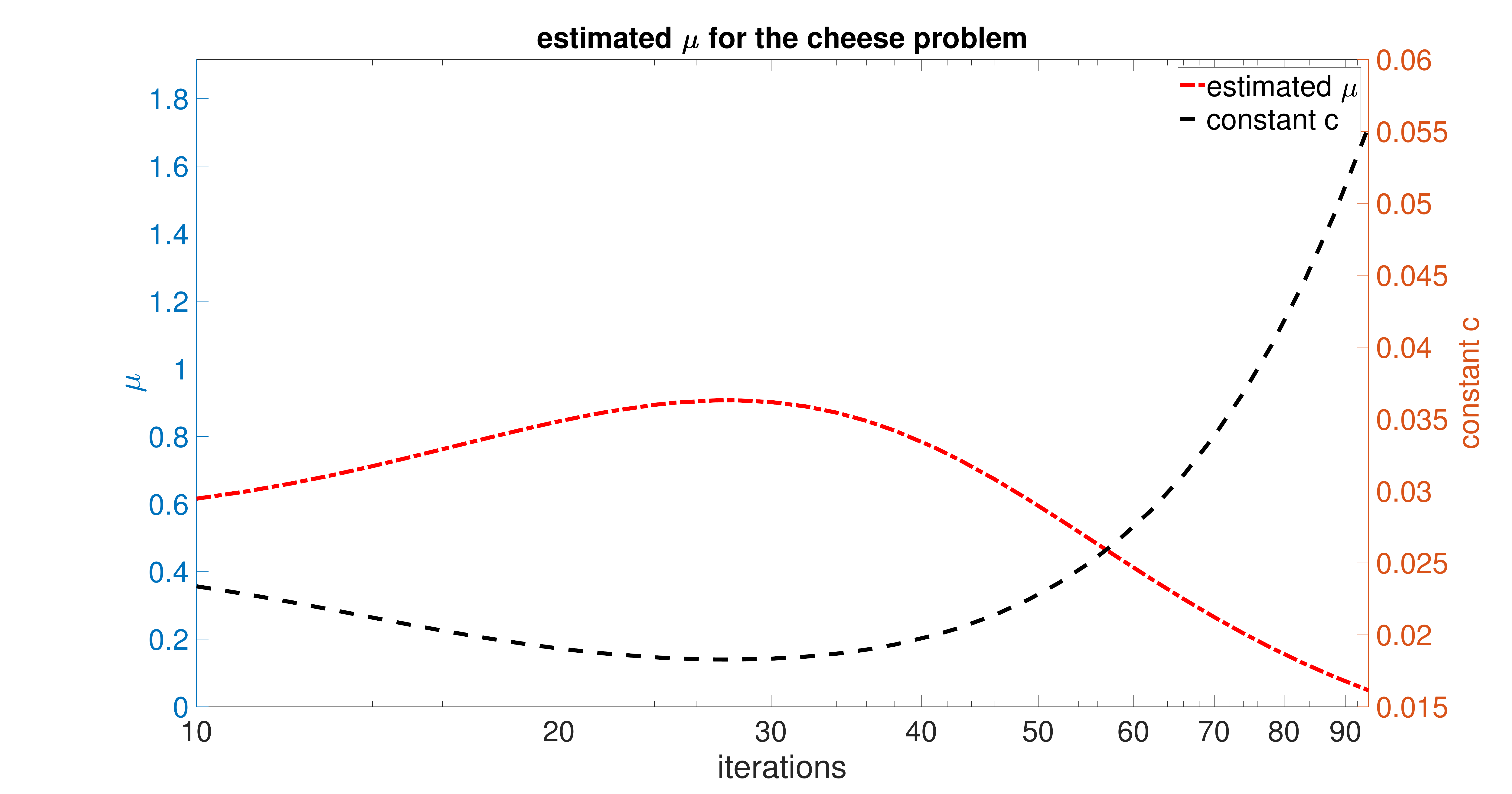}\caption{Demonstration of the method for the cheese problem. We estimate $\mu\approx 0.9$, but have no confirmation for this other than our method.}\label{fig:cheese}
\end{figure}

\begin{figure}
\includegraphics[width=\textwidth]{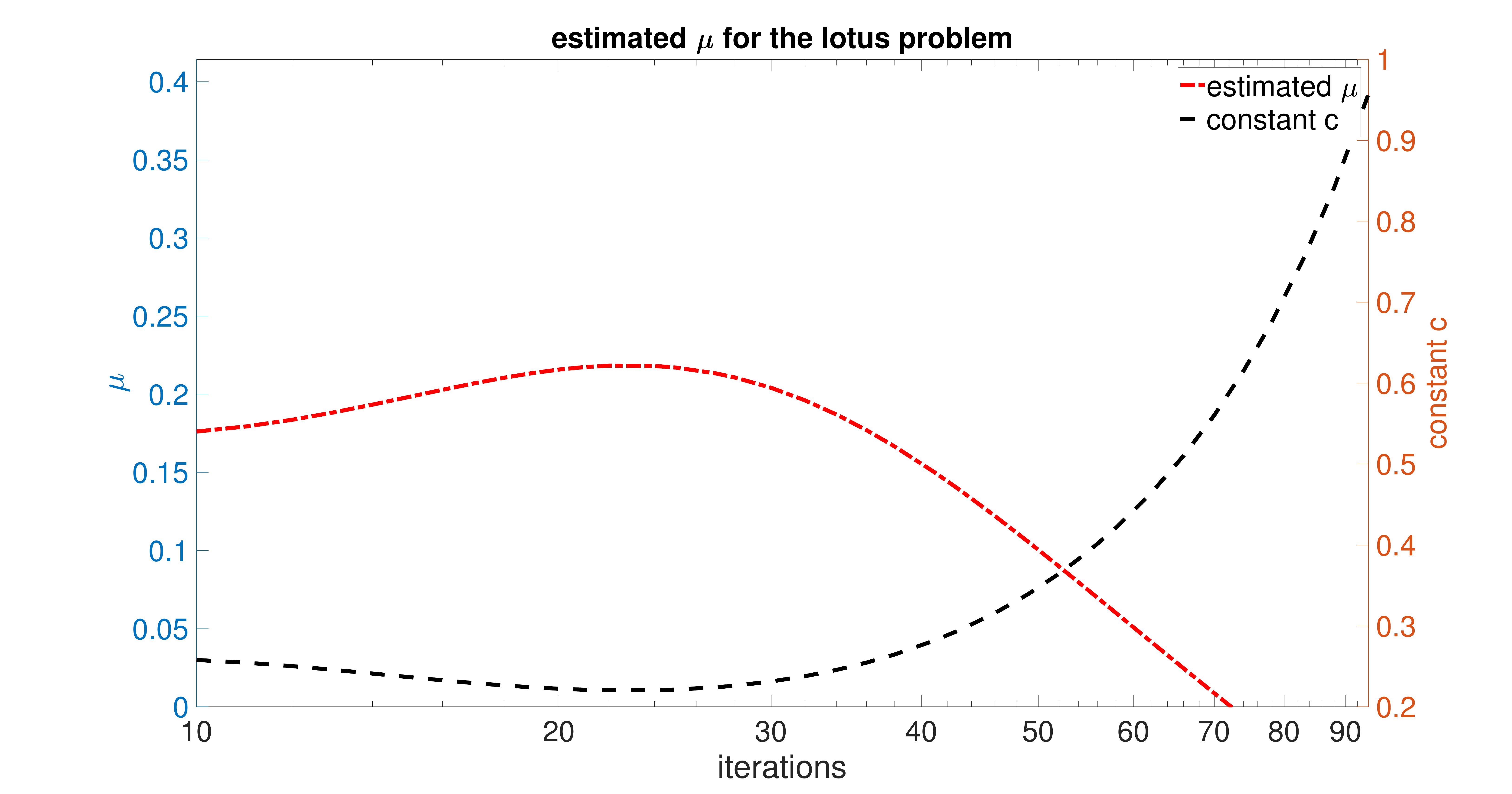}\caption{Demonstration of the method for the lotus problem. We estimate $\mu\approx 0.22$, but have no confirmation for this other than our method.}\label{fig:lotus}
\end{figure}

\begin{figure}
\includegraphics[width=\textwidth]{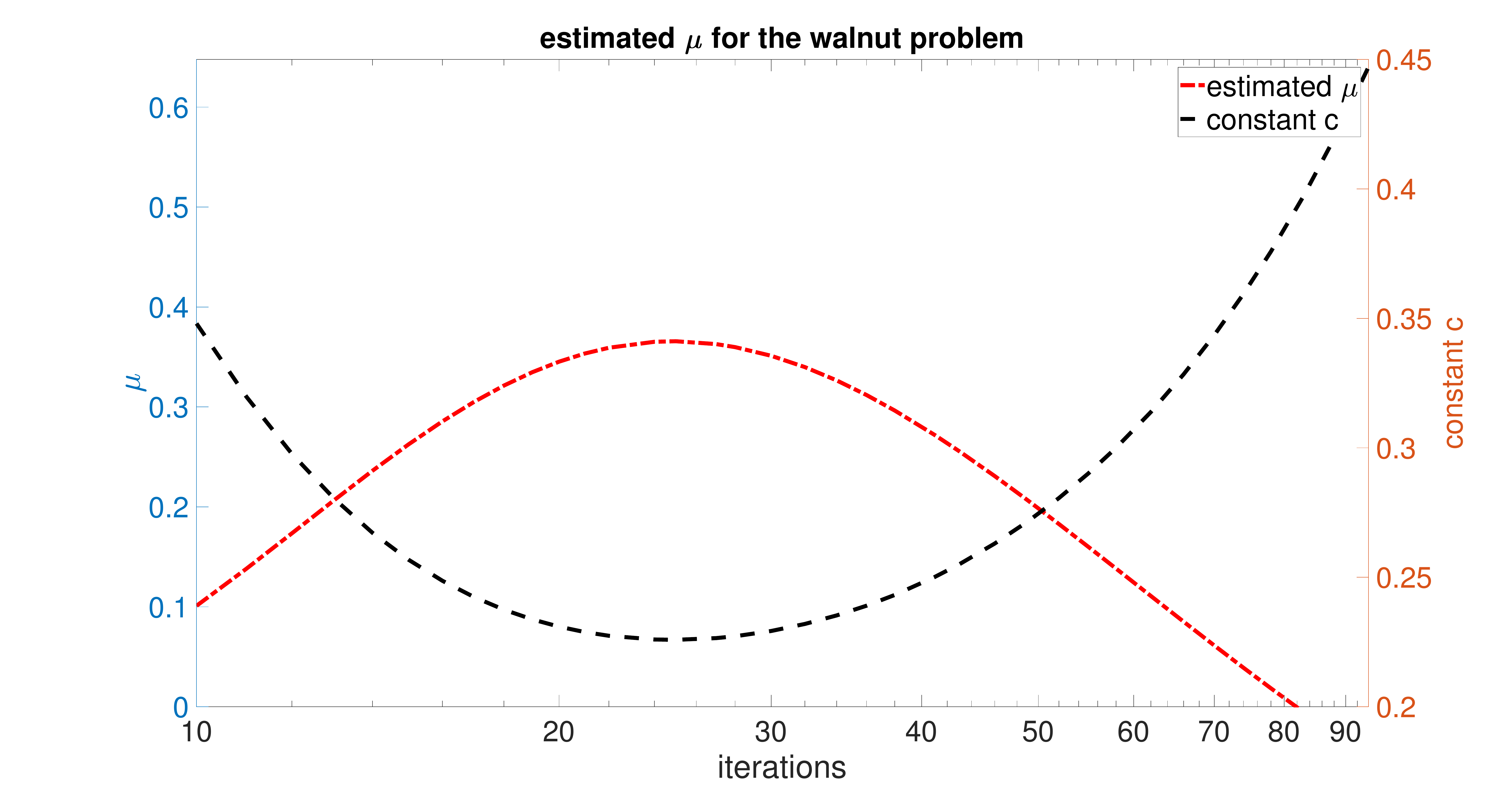}\caption{Demonstration of the method for the walnut problem. We estimate $\mu\approx 0.4$, but have no confirmation for this other than our method.}\label{fig:nuts}
\end{figure}

\begin{figure}
\includegraphics[width=0.32\linewidth]{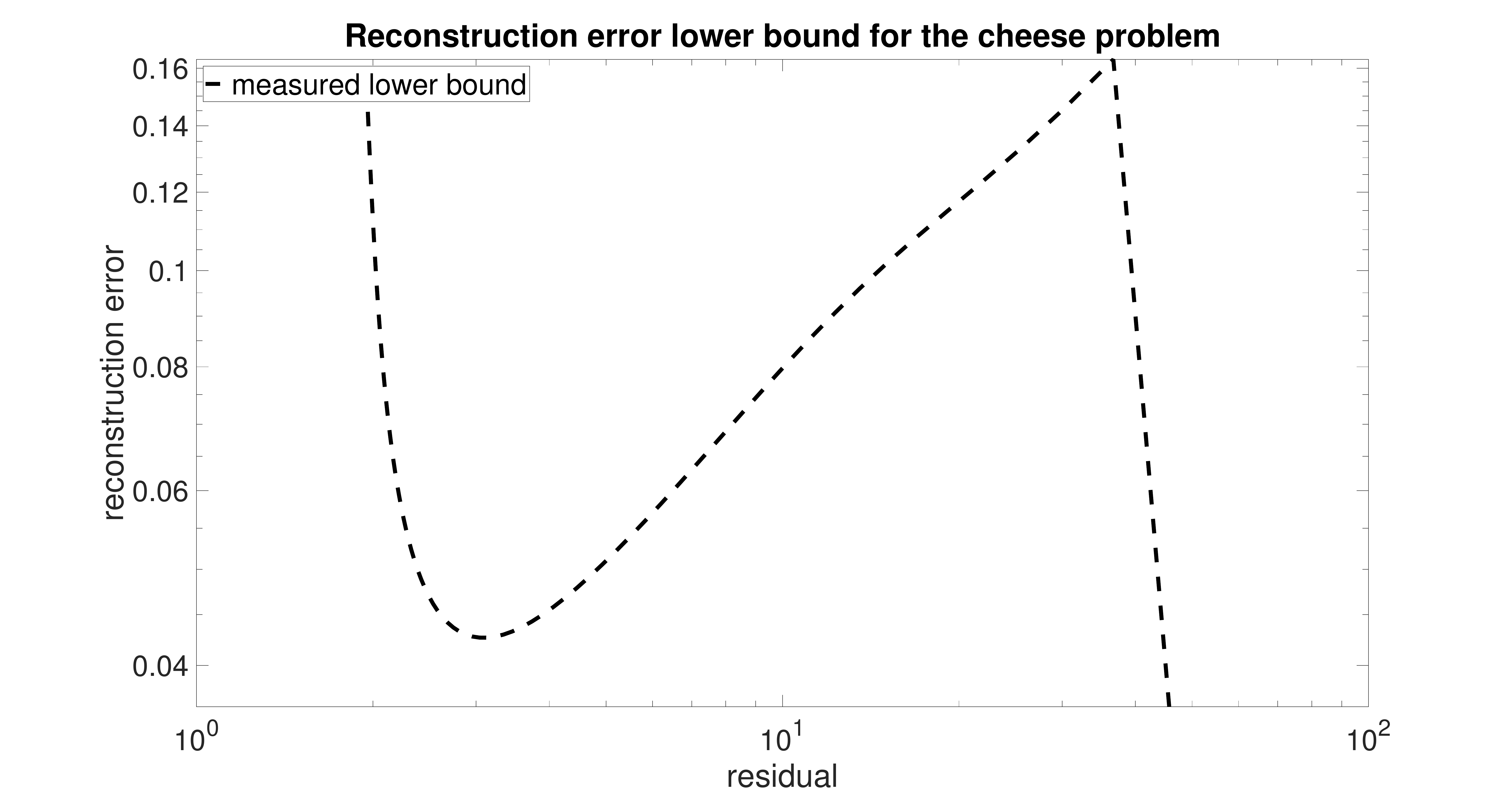}\includegraphics[width=0.32\linewidth]{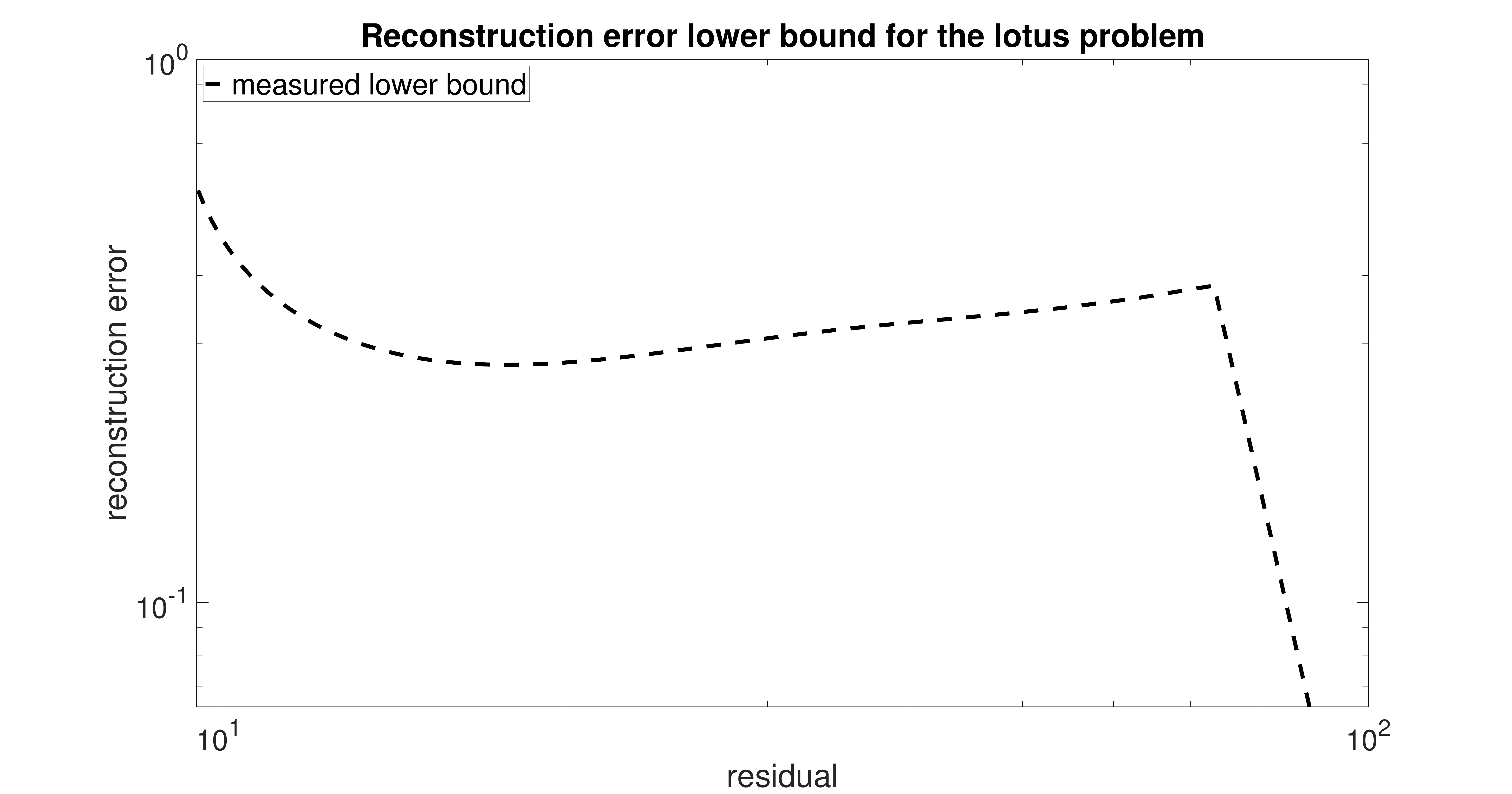}\includegraphics[width=0.32\linewidth]{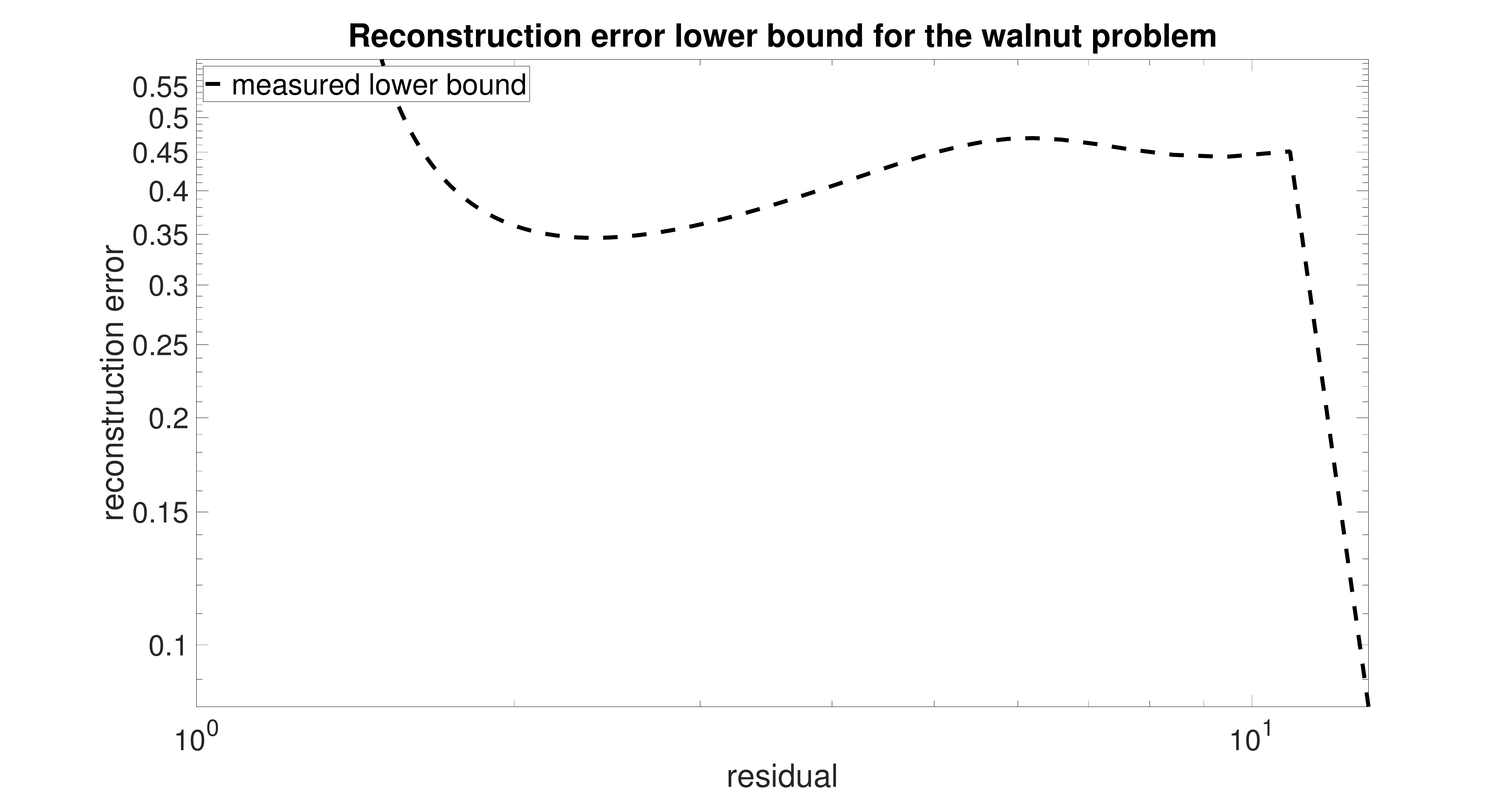}\caption{Observed lower bound from \eqref{eq:lb} for the three real data problems, from left to right: cheese, lotus, and walnut. For higher residuals, the cheese and lotus data have an approximately linear slope, hinting that the estimation of $\mu$ for the first iterations might be correct. The walnut data behaves completely irregular, hence the graph for the estimated $\mu$ in \ref{fig:nuts} is the worst among all 3 problems. For smaller residuals, the lower bound goes up, indicating that the noise takes over. This corresponds to the falling estimates of $\mu$ for the larger iterations in the corresponding Figures \ref{fig:cheese}, \ref{fig:lotus}, and \ref{fig:nuts}.}\label{fig:real_lower}
\end{figure}

\section*{Conclusion}
We have shown that the classical source condition for linear inverse problems implies a \L{}ojasiewicz inequality which is equivalent to convergence rates. We used the relation between residual and gradient that is implied by the \L{}ojasiewicz inequality to estimate the smoothness parameter in the source condition using Landweber iteration. This works reasonably well when the source condition is indeed fulfilled; the obtained estimates for the smoothness parameter agree with related knowledge. We have shown that even noisy and real-life data can be used as long as the noise is not too large. While there are many open problems, we believe that exploiting the \L{}ojasiewicz property may lead to an improvement in the numerical treatment of inverse problems and consider this paper as a first step in this direction.

\section*{Acknowledgments} Research supported by Deutsche Forschungsgemeinschaft (DFG-grant HO 1454/10-1).The author thanks  the anonymous referees for their comments that helped to improve this paper. The helpful comments and discussions with Bernd Hofmann (TU Chemnitz) and Oliver Ernst (TU Chemnitz) are gratefully acknowledged.

\bibliographystyle{unsrt}

\end{document}